\documentclass[10pt]{article}

\usepackage{amsfonts}
\usepackage{amssymb}
 \usepackage{amsthm,amsmath}
\usepackage{mathrsfs}
\usepackage{graphicx}
\usepackage{color}

\numberwithin{equation}{section}

\newtheorem{thm}{Theorem}[section]
\newtheorem{lem}{Lemma}[section]

\newtheorem{prop}{Proposition}[section]
\theoremstyle{definition}
\newtheorem{defn}{Definition}[section]
\theoremstyle{remark}
\newtheorem{rem}{Remark}[section]

\allowdisplaybreaks

\begin{document}

\title{Global well-posedness to the 3-D incompressible inhomogeneous
 Navier-Stokes equations with a class of large velocity}

\author{Cuili Zhai,  Ting Zhang \\
  Department of Mathematics, Zhejiang University,
Hangzhou 310027, China}
\date{}
\maketitle

\begin{abstract}
 In this article, we consider the global well-posedness to the 3-D
incompressible inhomogeneous Navier-Stokes equations with a class of
large velocity. More precisely, assuming $a_0 \in
\dot{B}_{q,1}^{\frac{3}{q}}(\mathbb{R}^3)$ and $u_0=(u_0^h,u_0^3)\in
\dot{B}_{p,1}^{-1+\frac{3}{p}}(\mathbb{R}^3)$ for $p,q \in (1,6)$ with $
\sup(\frac{1}{p}, \frac{1}{q})\leq\frac{1}{3}+ \inf (\frac{1}{p},
\frac{1}{q})$, we prove that if
$C\|a_0\|_{\dot{B}_{q,1}^{\frac{3}{q}}}^{\alpha}(\|u_0^3\|_{\dot{B}_{p,1}^{-1+\frac{3}{p}}}/{\mu}+1)\leq1$,
$\frac{C}{\mu}(\|u_0^h\|_{\dot{B}_{p,1}^{-1+\frac{3}{p}}}+\|
u_0^3\|_{\dot{B}_{p,1}^{-1+\frac{3}{p}}}^{1-\alpha}\|
u_0^h\|_{\dot{B}_{p,1}^{-1+\frac{3}{p}}}^{\alpha})\leq 1$, then the system
has a unique global solution $a\in\widetilde{\mathcal
{C}}([0,\infty);\dot{B}_{q,1}^{\frac{3}{q}}(\mathbb{R}^3))$,
$u\in\widetilde{\mathcal
{C}}([0,\infty);\dot{B}_{p,1}^{-1+\frac{3}{p}}(\mathbb{R}^3))\cap
L^1(\mathbb{R}^+;\dot{B}_{p,1}^{1+\frac{3}{p}}(\mathbb{R}^3))$. It
improves the recent result of M. Paicu, P. Zhang (J. Funct. Anal.
262 (2012) 3556-3584), where the exponent form of the initial
smallness condition is replaced by a polynomial form.

\textbf{Keywords}: Inhomogeneous Navier-Stokes equations; Well-posedness;
Littlewood-Paley theory.

\textbf{2010 AMS Subject Classification}: 35Q35, 76W05.
\end{abstract}

\section{Introduction.}
In this paper, we consider the global well-posedness of the
following 3-D incompressible inhomogeneous Navier-Stokes equations
with initial data in the critical Besov spaces
\begin{equation}\label{a}
 \left\{\begin{array}{l}
\partial_t{\rho}+ \displaystyle\mathrm{div}(\rho u)=0,
\ \ (t,x)\in\ \mathbb{R}^+\times\ \mathbb{R}^{3},\\
\partial_t(\rho u)+\displaystyle\mathrm{div}(\rho u\otimes u)-\displaystyle\mathrm{div}(2\mu\mathcal {M})+\nabla \Pi=0,\\
\displaystyle\mathrm{div}u=0,\\
\ \rho|_{t=0}=\rho_0,\ \ \rho u|_{t=0}=m_0,\\
\end{array}\right.
\end{equation}
where $\rho$, $u=(u_1,u_2,u_3)$ stand for the density and velocity
field of the fluid respectively, $\mathcal
{M}=\frac{1}{2}(\partial_iu_j+\partial_ju_i)$, $\Pi$ is a scalar
pressure function, and in general, the viscocity coefficient
$\mu(\rho)$ is a smooth, positive function on $[0,\infty)$. Such
system describes a fluid which is obtained by mixing two miscible
fluids that are incompressible and that have different densities. It
may also describe a fluid containing a melted substance. One may
check \cite{PL96} for the detailes derivation of this system.

 When $\mu(\rho)$ is independent of $\rho$, i.e. $\mu$ is a
positive constant, and $\rho_0$ is bounded away from 0, many authors
showed their investigations on this system, see \cite{AV73, AV90,
AV74, S90} etc. Kazhikov \cite{AV74} proved that (\ref a) has a
unique local smooth solution with regular initial data. In addition,
they proved the global existence of strong solutions to this system
for small data in three space dimensions and all data in two
dimensions. However, the uniqueness of both type weak solutions has
not be solved. Ladyzhenskaya and Solonnikov \cite{OA75} first
addressed the question of unique resolvability of (\ref a). And
recently, similar results were obtained by Danchin \cite{RD03, RD04}
in $\mathbb{R}^N$ with initial data in the almost critical spaces,
and it \cite {RD03} generalized the result by Fujita and Kato
\cite{FK64} denoted to the classical Navier-Stokes system.

 In general, $\mu=\mu(\rho)$, under some special assumption, a lot of results about
stability and well-posedness of Navier-Stokes equations were
received by many authors, such as \cite{A07, AGZ11, AGZ12, AGZ13,
AP07, PL89, PL96} etc. Diperna and Lions \cite{PL89, PL96} proved
the global existence of weak solutions to (\ref{a}) in any space
dimensions. Yet the uniqueness and regularities of such weak
solutions are big open questions even in two space dimension, as was
mentioned by Lions in \cite{PL96}. On the other hand, Abidi, Gui and
Zhang \cite{AGZ11} investigated the large time decay and stability
to any given global smooth solutions of (\ref{a}), which in
particular implies the global well-posedness of 3-D inhomogeneous
Navier-Stokes equations with axi-symmetric initial data and without
swirl for the initial velocity field provided that the initial
density is close enough to a positive constant.

 When the density $\rho$ is away from zero, we denote by
$a\stackrel{def}{=}\frac{1}{\rho}-1$ and
$\widetilde{\mu}(a)\stackrel{def}{=}\mu(\frac{1}{1+a})$, then the
system (\ref{a}) can be equivalently reformulated as
\begin{equation*}
 (\mbox{INS})\ \ \ \ \left\{\begin{array}{l}
\partial_t a+ u\cdot\nabla a=0,
\ \ (t,x)\in\ \mathbb{R}^+\times\ \mathbb{R}^N,\\
\partial_t u+u\cdot\nabla u+(1+a)(\nabla \Pi-\displaystyle\mathrm{div}(2\widetilde{\mu}(a)\mathcal {M}))=0,\\
\displaystyle\mathrm{div}u=0,\\
(a,u)|_{t=0}=(a_0,u_0).
\end{array}\right.
\end{equation*}

 In \cite{A07}, Abidi proved if $1<p<2N$,
$0<\underline{\mu}<\widetilde{\mu}(a)$, $u_0\in
\dot{B}_{p,1}^{\frac{N}{p}-1}(\mathbb{R}^N)$ and $a_0\in
\dot{B}_{p,1}^{\frac{N}{p}}(\mathbb{R}^N)$, then (INS) has a global
solution provided that
$\|a_0\|_{\dot{B}_{p,1}^{\frac{N}{p}}}+\|u_0\|_{\dot{B}_{p,1}^{\frac{N}{p}-1}}\leq
c_0$ for some $c_0$ sufficiently small. Furthermore, the solution
thus obtained  is unique if $1<p\leq N$. And this result generalized
the corresponding results in \cite{RD03, RD04}.

 For simplicity, in this paper,  we just take $\mu(\rho)=\mu$
and the space dimension $N=3$. Thus (INS) becomes
\begin{equation}\label{equation}
 \left\{\begin{array}{l}
\partial_t{a}+u\cdot\nabla a=0,
\ \ (t,x)\in\ \mathbb{R}^+\times\ \mathbb{R}^{3},\\
\partial_t{u}+u\cdot\nabla u+(1+a)(\nabla\Pi-\mu\Delta u)=0,\\
\displaystyle\mathrm{div} u=0,\\
\ (a, u)|_{t=0}=(a_0, u_0).
\end{array}\right.
\end{equation}

 Before we present our main result in this paper, let us recall
the following results from Abidi, Paicu \cite{AP07},  Danchin,   Mucha \cite{DM12} and Paicu, Zhang
\cite{PP12}. We denote
$$
E_{p,q,T}\stackrel{def}{=}\left\{(a,u,\nabla \Pi)\left|\begin{array}{l}
a\in \widetilde{\mathcal
{C}}_T(\dot{B}_{q,1}^{\frac{3}{q}}(\mathbb{R}^3)), u\in
\widetilde{\mathcal
{C}}_T(\dot{B}_{p,1}^{-1+\frac{3}{p}}(\mathbb{R}^3))\cap
L_T^1(\dot{B}_{p,1}^{1+\frac{3}{p}}(\mathbb{R}^3))\\
 \nabla\Pi\in
L_T^1(\dot{B}_{p,1}^{-1+\frac{3}{p}}(\mathbb{R}^3))
\end{array}\right.
\right\},
$$
whereas
$$\widetilde{\mathcal
{C}}_T(\dot{B}^{s}_{p,r}(\mathbb{R}^3))\stackrel{def}{=}\mathcal
{C}([0,T]; \dot{B}^{s}_{p,r}(\mathbb{R}^3))\cap \widetilde{L}^\infty(0,T;
\dot{B}^{s}_{p,r}(\mathbb{R}^3)).$$
For simplify, we denote $E_{p,q}$  when $T=\infty$.
\begin{thm}[see \cite{AP07,DM12}]\label{abidi theorem}
Let $q, p$ satisfy $q, p \in (1,\infty)$ so that $\sup
(\frac{1}{p},\frac{1}{q})\leq \frac{1}{3}+\inf
(\frac{1}{p},\frac{1}{q})$ and $\frac{1}{p}+\frac{1}{q}>
\frac{1}{3}$. Let $a_0 \in \dot{B}_{q,1}^{\frac{3}{q}}(\mathbb{R}^3),
u_0\in \dot{B}_{p,1}^{-1+\frac{3}{p}}(\mathbb{R}^3)$ with
$\|a_0\|_{\dot{B}_{q,1}^{\frac{3}{q}}} \leq c$ for some sufficiently small
$c$, then the system (\ref{equation}) has a unique local solution $(a,u,\nabla \Pi)$
on $[0,T]$ such that $(a,u,\nabla \Pi)\in E_{p,q,T}$.
 Moreover, if
$\|u_0\|_{\dot{B}_{p,1}^{-1+\frac{3}{p}}} \leq c'\mu$ for $c'$ small
enough, then the solution exists on $[0, +\infty)$.
\end{thm}

 Indeed, Abidi and Paicu \cite{AP07} only proved that the solution
is unique when $\frac{1}{p}+\frac{1}{q}\geq \frac{2}{3}$, and very
recently, Danchin and Mucha \cite{DM12} improved the uniqueness for
$\frac{1}{p}+\frac{1}{q}>\frac{1}{3}$ through lagrangian approach.

 Motivated by \cite{GP10, PP11, ZT09}, M. Paicu, P. Zhang
proved the following theorem in \cite{PP12} by applying the
technology of a weighted Chemin-Lerner type norm.
\begin{thm}[see \cite{PP12}] \label{theorem zhangping}
Let $1<q \leq p<6$ with $\frac{1}{q}-\frac{1}{p} \leq \frac{1}{3}$.
There exist positive constants $c_0$ and $C_0$ such that, for any
data $a_0 \in \dot{B}_{q,1}^{\frac{3}{q}}(\mathbb{R}^3)$ and $u_0 =(u_0^h,
u_0^3)\in \dot{B}_{p,1}^{-1+\frac{3}{p}}(\mathbb{R}^3)$ verifying
\begin{equation}\label{pp condition}
  \eta\stackrel{def}{=}(\mu\|a_0\|_{\dot{B}_{q,1}^{\frac{3}{q}}}+\|u_0^h\|_{\dot{B}_{p,1}^{-1+\frac{3}{p}}})\exp\{C_0\|u_0^3\|_{\dot{B}_{p,1}^{-1+\frac{3}{p}}}^2/{\mu^2}\}
\leq c_0\mu,
\end{equation}
the system (\ref{equation}) has a unique global solution $(a,u,\nabla \Pi)\in E_{p,q}$.
\end{thm}
\begin{center}
  \includegraphics[width=2.5in]{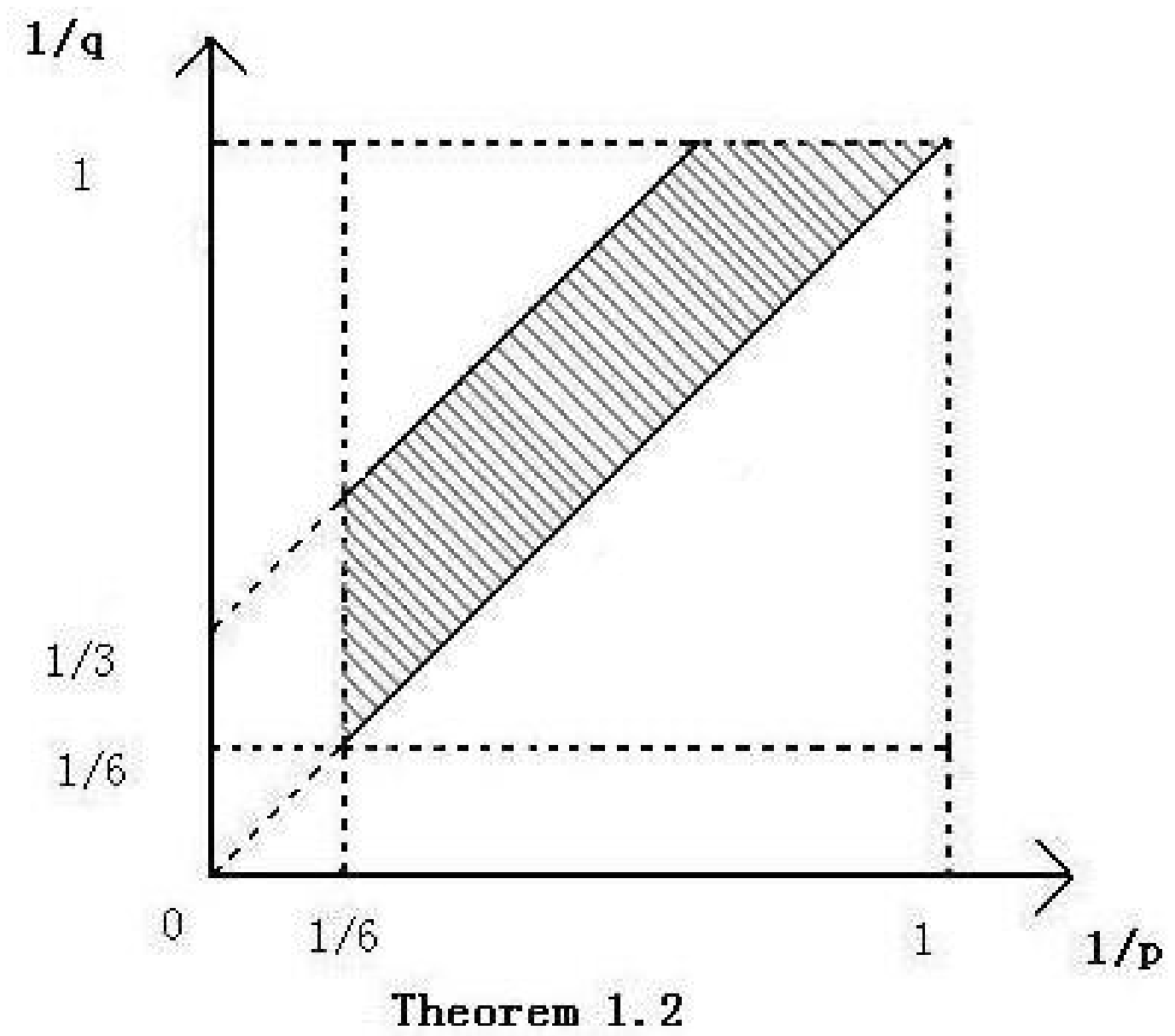}\includegraphics[width=2.5in]{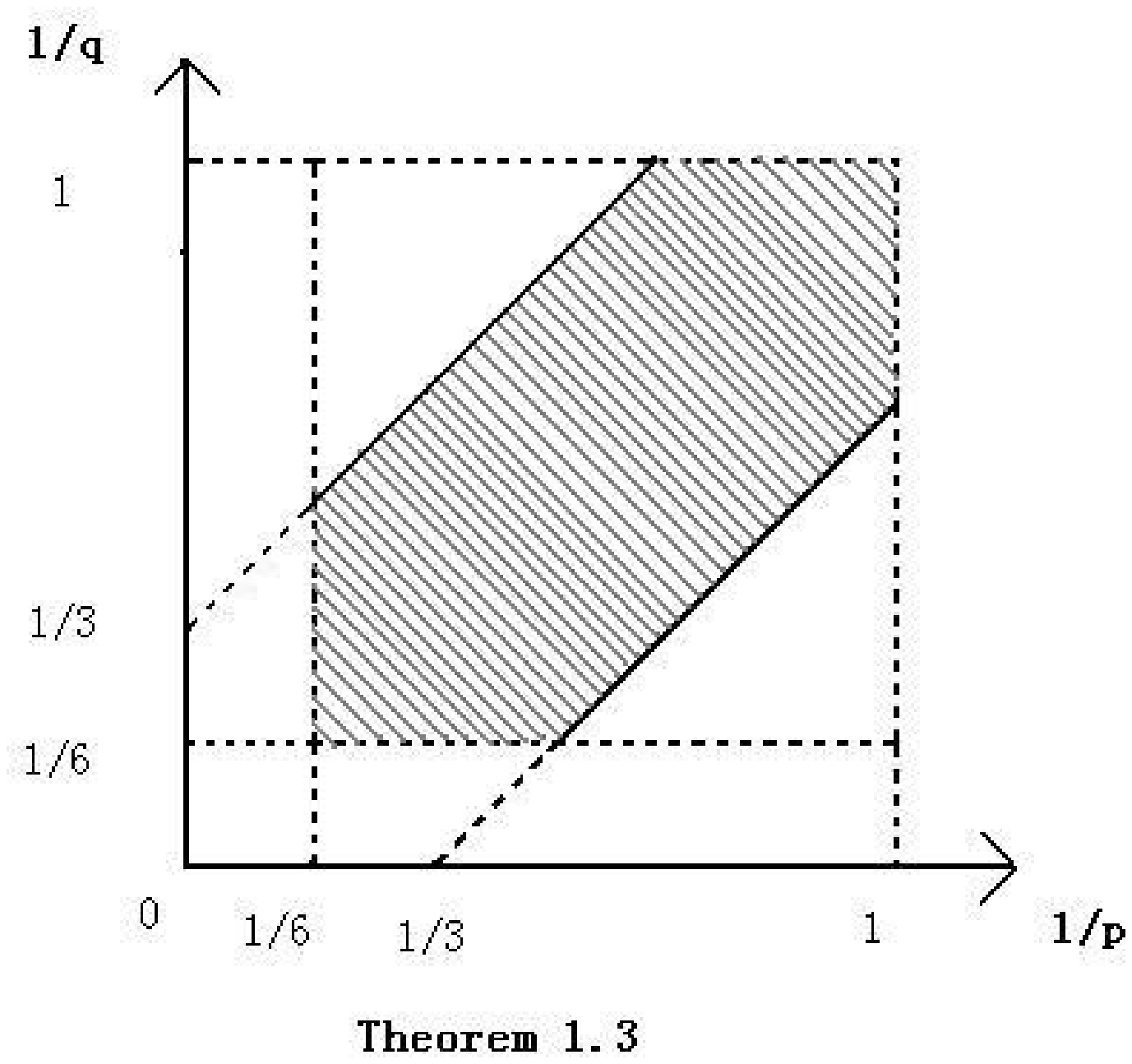}
\end{center}

 We note that, for the classical Navier-Stokes equations, the
equation on $u^3$ is a linear system for fixed $u^h=(u_1, u_2)$,
while the system on $u^h$ is nonlinear. By using the technology of a
weighted Chemin-Lerner type Besov spaces norm in \cite{PP11} or
Gronwall's inequality in \cite{GP10,ZT09}, the form of small initial
conditions liking (\ref{pp condition}) were obtained, which implies
that the third component of the initial velocity field can be large.
And by using the algebraical structure of (\ref{equation}):
$\displaystyle\mathrm{div} u=0$, they \cite{PP12} proved the Theorem
\ref{theorem zhangping} by energy estimates on the horizontal
components and the vertical component of the velocity field
respectively.

 In this paper, we are going to relax the smallness condition in
Theorems \ref{abidi theorem} and \ref{theorem zhangping}, so that (\ref{equation}) still has
a unique global solution.  Now
we present the first main result in this paper:
\begin{thm}\label{main}
Let $p,q$ satisfy $p,q \in (1,6)$ so that $ \sup(\frac{1}{p},
\frac{1}{q})\leq\frac{1}{3}+ \inf (\frac{1}{p}, \frac{1}{q})$. There
exists a positive constant $C$ such that, for any data $a_0 \in
\dot{B}_{q,1}^{\frac{3}{q}}(\mathbb{R}^3)$ and $u_0=(u_0^h,u_0^3)\in
\dot{B}_{p,1}^{-1+\frac{3}{p}}(\mathbb{R}^3)$ verifying
\begin{equation}\label{data condition0}
C\|a_0\|_{\dot{B}_{q,1}^{\frac{3}{q}}}^{\alpha}(\|u_0^3\|_{\dot{B}_{p,1}^{-1+\frac{3}{p}}}/{\mu}+1)\leq1
\end{equation}
\begin{equation}\label{data condition1}
\frac{C}{\mu}(\|u_0^h\|_{\dot{B}_{p,1}^{-1+\frac{3}{p}}}+\|
u_0^3\|_{\dot{B}_{p,1}^{-1+\frac{3}{p}}}^{1-\alpha}\|
u_0^h\|_{\dot{B}_{p,1}^{-1+\frac{3}{p}}}^{\alpha})\leq 1,
\end{equation}
whereas
\begin{equation}\label{definition alpha}
 \alpha=\left\{\begin{array}{l}
\frac{1}{p},\ \ 1<p<5\\
\varepsilon,\ \ 5\leq p<6\\
\end{array}\right.
\end{equation}
for $0<\varepsilon<\frac{6}{p}-1$, the system (\ref{equation}) has a
unique global solution $(a,u,\nabla \Pi)\in E_{p,q}$.
\end{thm}
\begin{rem}
Motivated by \cite{ZT14}, using Gagliardo-Nirenberg
inequality, we obtain $\|u^3\|_{L_v^\infty}\leq
C\|u^3\|_{L_v^p}^{1-\frac{1}{p}}\|\displaystyle\mathrm{div}_h
u^h\|_{L_v^p}^{\frac{1}{p}}$, and one can search for details in Lemma \ref{u^3 (1)} in
the second section. This implies that the velocity is large in one
direction with the $L^p$ framework functional space, but is small in
the $L^{\infty}$ framework functional space.
\end{rem}
\begin{rem}
We assert that our theorem remains to be true in the case when the
viscosity coefficient depends on the density by a regular function
$\mu(\rho)$ with $\mu(\rho)\geq \mu>0$. In this case, we just need a
small modification of the proof to Theorem \ref{main} by using the
fact that: for any positive $s$, we have
$\|\widetilde{\mu}(a)-\widetilde{\mu}(0)\|_{\dot{B}_{q,1}^s}\leq
C(1+\|a\|_{L^{\infty}})^{[s]+1}\|a\|_{\dot{B}_{q,1}^s}$, where
$\widetilde{\mu}(a)\stackrel{def}{=}\mu(\frac{1}{1+a})$.
\end{rem}
\begin{rem}
We can also have a version of Theorem \ref{main} in any space
dimension. Just for a clear presentation, we choose to work in the
three space dimension case here.
\end{rem}
\begin{rem}
About ill-posedness of the classical impressible Navier-Stokes
equations with $\rho=1$, Bourgain-Pavlovie \cite{BP} and Germain
\cite{G} proved the ill-posedness in the largest critical space
$\dot{B}_{\infty,\infty}^{-1}$. Motivated by \cite{BP}, Chen, Miao, Zhang
\cite{CMZ13} proved that the 3-D baratropic Navier-Stokes equations
is ill-posed for the initial density and velocity belonging to the
critical Besov spaces
$(\dot{B}_{p,1}^{\frac{3}{p}}+\bar{\rho},\dot{B}_{p,1}^{\frac{3}{p}-1})$ for
$p>6$, here $\bar{\rho}$ is a positive constant. In the future, we
will work in ill-posedness of the incompressible inhomogeneous
Navier-Stokes system in critical Besov spaces.
\end{rem}

 Very recently,
Danchin and Mucha \cite{DM12} also obtained a more general result by
considering very rough densities in some multiplier spaces on the
Besov spaces $\dot{B}_{p,1}^{-1+\frac{3}{p}}(\mathbb{R}^3)$. In particuar,
they are able to consider the physical case of mixture of fluids
with piecewise constant density. We emphasize that the main feature
of the density used in this theorem is to be a multiplier on the
velocity space. This allows to define the nonlinear terms containing
products between the density and the velocity in the system
(\ref{equation}). And Motivated by \cite{DM12,HPP12}, we can also
replace the $\|a_0\|_{\dot{B}_{q,1}^{\frac{3}{q}}}$ in the smallness
condition (\ref{data condition0}) by
$\|a_0\|_{\mathscr{M}(\dot{B}_{p,1}^{-1+\frac{3}{p}})}$ and prove a
similar version of Theorem \ref{main}:
\begin{thm}\label{main theorem}
Let $\frac{3}{2}<p<6$. Let $a_0\in
\mathscr{M}(\dot{B}_{p,1}^{-1+\frac{3}{p}}(\mathbb{R}^3))$ and
$u_0=(u_0^h,u_0^3)\in \dot{B}_{p,1}^{-1+\frac{3}{p}}(\mathbb{R}^3)$. Then
there exists a positive constant $C$ such that if
\begin{equation*}
C\|a_0\|_{\mathscr{M}(\dot{B}_{p,1}^{-1+\frac{3}{p}})}^{\alpha}(\|u_0^3\|_{\dot{B}_{p,1}^{-1+\frac{3}{p}}}/{\mu}+1)\leq1
\end{equation*}
\begin{equation}\label{data condition2}
\frac{C}{\mu}(\|u_0^h\|_{\dot{B}_{p,1}^{-1+\frac{3}{p}}}+\|
u_0^3\|_{\dot{B}_{p,1}^{-1+\frac{3}{p}}}^{1-\alpha}\|
u_0^h\|_{\dot{B}_{p,1}^{-1+\frac{3}{p}}}^{\alpha})\leq 1,
\end{equation}
whereas $\alpha$ defined as (\ref{definition alpha}), then
(\ref{equation}) has a unique global solution $a\in
\widetilde{\mathcal {C}}^{\infty}(\mathbb{R}^+;
\mathscr{M}(\dot{B}_{p,1}^{-1+\frac{3}{p}})(\mathbb{R}^3))$ and
$u\in\widetilde{\mathcal
{C}}([0,\infty);\dot{B}_{p,1}^{-1+\frac{3}{p}}(\mathbb{R}^3))\cap
L^1(\mathbb{R}^+;\dot{B}_{p,1}^{1+\frac{3}{p}}(\mathbb{R}^3))$.
\end{thm}

\begin{rem}
Using  basic continuity results for the paraproduct operator (see
  \cite{RD11}), one can obtain that any space $L^{\infty}(\mathbb{R}^3)\cap
  \dot{B}_{q,\infty}^{\frac{3}{q}}(\mathbb{R}^3)$ with $q$ satisfying
  \begin{equation}\label{pq condition}
  \frac{1}{q}+\frac{1}{p}>\frac{1}{3}\ \ and \ \
  \frac{1}{p}-\frac{1}{q}<\frac{1}{3}
  \end{equation}
  embeds $\mathscr{M}(\dot{B}_{p,1}^{-1+\frac{3}{p}}(\mathbb{R}^3))$. It contains
  characteristic functions of $C^1$-bounded domains whenever $p>2$
  (see the proof of Lemma A.7 in \cite{DM12}).  Hence our result
  applies to a mixture of fluids, which is of great physical
  interest.
\end{rem}
\textbf{Scheme of the proof and organization of the paper.} In the
second section, we shall collect some basic facts on
Littlewood-Paley analysis. In the third section , we prove theorem
\ref{main} in case when $p\leq q$. Finally in the last section, we
shall prove the case of $p>q$. And in the Appendix, we give the
proof of Theorem \ref{main theorem}.

 Let us complete this section with the notations we are going to
this context.\\
\textbf{Notations.} Let $A, B$ be two operators, we denote
$[A;B]=AB-BA$, the commutator between $A$ and $B$. For $a\lesssim
b$, we mean that there is a uniform constant $C$, which may be
different on different lines, such that $a\leq Cb$. We shall denoted
by $(a|b)$ the $L^2(\mathbb{R}^3)$ inner product of $a$ and
 $b$. $(d_j)_{j\in\mathbb{Z}}$ will be a generic element of
$\ell^1(\mathbb{Z})$ so that $d_j\geqslant0$ and
$\Sigma_{j\in\mathbb{Z}}d_j=1.$

 For $X$ a Banach space and $I$ an interval of $\mathbb{R}$, we
denote by $\mathcal {C}(I,X)$ the set of continuous functions on $I$
with values in $X$, and by $L^p(I;X)$ stands for the set of
measurable functions on $I$ with values in $X$, such that
$t\longmapsto\|f(t)\|_X$ belongs to $L^p(I)$. Finally we denote
$L_T^p(L_h^q(L_v^r))$ the space
$L^p([0,T];L^q(\mathbb{R}_{x_1}\times\mathbb{R}_{x_2};L^r(\mathbb{R}_{x_3})))$.
\section{Preliminaries}
The proof of Theorem \ref{main} requires the Littlewood-Paley
decomposition. Let us briefly explain how it may be built in the
case $x\in\mathbb{R}^3$ (see e.g.\cite{RD11}). Let $\varphi$ be a
smooth function supported in the ring $\mathcal
{C}\stackrel{def}{=}\{\xi\in\mathbb{R}^3,\frac{3}{4}\leq|\xi|\leq\frac{8}{3}\}$
and such that\\
$$
\sum\limits_{j\in\mathbb{Z}}\varphi(2^{-j}\xi)=1  \ \  \textrm{for} \ \
|\xi|\neq0.
$$
Then for $u \in \mathcal {S}'(\mathbb{R}^3)$, we set\\
$$
\forall j \in\mathbb{Z},\ \
\Delta_ju\stackrel{def}{=}\varphi(2^{-j}D)u\ \ \textrm{and}\ \
S_ju\stackrel{def}{=}\sum\limits_{\ell\leq j-1}\Delta_\ell u.
$$

 The Besov space can be characterized in virtue of the
Littlewood-Paley decomposition. Let $\mathcal {S}'_h(\mathbb{R}^3)$
be the space of tempered distributions $u$ such that
$$
\lim\limits_{\lambda\rightarrow\infty}\|\theta(\lambda
D)u\|_{L^{\infty}}=0\ \ \textrm{for  any} \ \ \theta\in\mathcal
{D}(\mathbb{R}^3),
$$
where $\mathcal {D}(\mathbb{R}^3)$ is the space of smooth compactly
supported functions on $\mathbb{R}^3$. Moreover, the
Littlewood-Paley decomposition satisfies the property of almost
orthogonality:
$$
\Delta_k\Delta_ju\equiv0\ \ \textrm{if}\ \ |k-j|\geq2\ \ \textrm{and} \ \
\Delta_k(S_{j-1}u\Delta_ju)\equiv0\ \ if\ \ |k-j|\geq5.
$$
We recall now the definitions of homogeneous Besov spaces and Chemin-Lerner-type spaces
$\widetilde{L}_T^{\lambda}(\dot{B}_{p,r}^s(\mathbb{R}^3))$ from
\cite{RD11}.
\begin{defn} \label{besov definition}
Let $(p,r)\in[1,+\infty]^2, s\in \mathbb{R}$ and $u\in \mathcal
{S}'_h(\mathbb{R}^3)$, we set
$$
\|u\|_{\dot{B}_{p,r}^s}\stackrel{def}{=}\{2^{qs}\|\Delta_q u\|_{L
^p}\}_{\ell^r}.
$$
$\bullet$ For $s<\frac{3}{p}$ (or $s=\frac{3}{p}$ if $r=1$), we
define $\dot{B}_{p,r}^s(\mathbb{R}^3)\stackrel{def}{=}\{u\in\mathcal
{S}'_h(\mathbb{R}^3)|\ \ \|u\|_{\dot{B}_{p,r}^s}<\infty\}$.\\
$\bullet$ If $k\in\mathbb{N}$ and $\frac{3}{p}+k\leq
s<\frac{3}{p}+k+1$ (or $s=\frac{3}{p}+k+1$ if $r=1$), then
$\dot{B}_{p,r}^s(\mathbb{R}^3)$ is defines as the subset of distributions
$u\in\mathcal {S}'_h(\mathbb{R}^3)$ such that $\partial^{\beta}u\in
\dot{B}_{p,r}^{s-k}(\mathbb{R}^3)$ whenever $|\beta|=k$.
\end{defn}

\begin{defn}
Let $s\in\mathbb{R}$, $(r, \lambda, p)\in[1, +\infty]^3$ and
$T\in(0, +\infty]$. We define
$\widetilde{L}_T^{\lambda}(\dot{B}_{p,r}^s(\mathbb{R}^3))$ as the
completion of $C([0,T];\mathcal {S}(\mathbb{R}^3))$ by the norm
$$
\|f\|_{\widetilde{L}_T^{\lambda}(\dot{B}_{p,r}^s)}\stackrel{def}{=}(\sum\limits_{q\in\mathbb{Z}}2^{qrs}(\int_0^T\|\Delta_qf(t)\|_{L^p}^{\lambda}dt)^{\frac{r}{\lambda}})^{\frac{1}{r}}<\infty,
$$
with the usual change if $r=\infty$. For short, we just denote this
space by $\widetilde{L}_T^{\lambda}(\dot{B}_{p,r}^s)$.
\end{defn}

 As we shall frequently use the anisotropic Bernstein inequalities. For
the convenience of the reader, we recall the following Bernstein
type lemma from \cite{RD11,Paicu05}.
\begin{lem}\label{bernstein}
Let $\mathcal {C}$ be a ring of $\mathbb{R}^3$ and $N\in
\mathbb{N}$. There exists a constant $C$ such that for any
homogeneous function $\sigma$ of degree $m$ smooth outside of $0$
and all $1\leq a\leq b\leq\infty$, we have \\
If the support of $\hat{u}$ is included in $2^k\mathcal {C}$, then
$$
C^{-1-N}2^{kN}\|u\|_{L^a}\leq\sup_{|\alpha|=N}\|\partial^{\alpha}u\|_{L^a}
\leq C^{1+N}2^{kN}\|u\|_{L^a}.
$$
If the support of $\hat{u}$ is included in $2^k\mathcal {C}$, then
$$\|\sigma(D)u\|_{L^b}\leq
C_{\sigma,m}2^{km+3k(\frac{1}{a}-\frac{1}{b})}\|u\|_{L^a}.
$$
Furthermore, let $\mathcal {B}_h$ (resp. $\mathcal {B}_v$) a ball of
$\mathbb{R}_h^2$ (resp. $\mathbb{R}_v$), let $1\leq p_2\leq
p_1\leq\infty$, $1\leq q_2\leq q_1\leq\infty$, we have\\
If the support of $\hat{u}$ is included in $2^j\mathcal {B}_h$, then
$$
\|\partial_{x_h}^{\alpha}
u\|_{L_h^{p_1}(L_v^{q_1})}\lesssim2^{j(|\alpha|+2(\frac{1}{p_2}-\frac{1}{p_1}))}\|u\|_{L_h^{p_2}(L_v^{q_1})}.
$$
If the support of $\hat{u}$ is included in $2^{\ell}\mathcal {B}_v$,
then
$$
\|\partial_{x_3}^{\beta}
u\|_{L_h^{p_1}(L_v^{q_1})}\lesssim2^{\ell(\beta+(\frac{1}{q_2}-\frac{1}{q_1}))}\|u\|_{L_h^{p_1}(L_v^{q_2})}.
$$
\end{lem}

 In the sequel, we shall frequently use Bony's decomposition
from \cite{RD11} in the homogeneous context:
\begin{equation}\label{bony}
  uv=T_uv+T_vu+R(u,v) \ \ or\ \ uv=T_uv+\mathcal {R}(u,v),
\end{equation}
where
\begin{eqnarray*}
  T_uv\stackrel{def}{=}\sum\limits_{j\in\mathbb{Z}}S_{j-1}u\Delta_jv,&\mathcal {R}(u,v)\stackrel{def}{=}\sum\limits_{j\in\mathbb{Z}}\Delta_juS_{j+2}v,\\
R(u,v)\stackrel{def}{=}\sum\limits_{j\in\mathbb{Z}}\Delta_ju\widetilde{\Delta}_jv,
& \textrm{and}\ \
\widetilde{\Delta}_jv\stackrel{def}{=}\sum\limits_{|j'-j|\leq1}\Delta_{j'}v.
\end{eqnarray*}
Finally, for the sake of completeness, we recall the following
product laws from Lemma $2.2$ in \cite{PP12}:
\begin{lem}\label{product law}
Let $p_2\geq p_1\geq1$,
$s_1\leq\frac{3}{p_1},s_2\leq\frac{3}{p_2}$ with
$s_1+s_2>3\max(0,\frac{1}{p_1}+\frac{1}{p_2}-1)$, $a\in
\dot{B}_{p_1,1}^{s_1}(\mathbb{R}^3),b\in \dot{B}_{p_2,1}^{s_2}(\mathbb{R}^3)$.
Then $ab\in \dot{B}_{p_2,1}^{s_1+s_2-\frac{3}{p_1}}(\mathbb{R}^3)$, and
there holds
\begin{equation*}
\|ab\|_{\dot{B}_{p_2,1}^{s_1+s_2-\frac{3}{p_1}}}\lesssim\|a\|_{\dot{B}_{p_1,1}^{s_1}}\|b\|_{\dot{B}_{p_2,1}^{s_2}}.
\end{equation*}
\end{lem}

In this paper, however, we need more general product laws. As an
application of Littlewood-Paley theory, we only present the
following product laws in Besov spaces, which will be used in the
sequel. One may check \cite{RD11} for more general product laws in
this respect.

\begin{lem}\label{product law 2}
 Let $1\leq p\leq q$,    $s\leq\frac{3}{q}$ with $s+\frac{3}{q}>3\max\{0,
\frac{1}{p}+\frac{1}{q}-1\}$, $a\in
\dot{B}_{q,1}^{\frac{3}{q}}(\mathbb{R}^3), b\in \dot{B}_{p,1}^s(\mathbb{R}^3)$.
Then $ab\in \dot{B}_{p,1}^s(\mathbb{R}^3)$, and there holds
\begin{equation*}
\|ab\|_{\dot{B}_{p,1}^s}\lesssim\|a\|_{\dot{B}_{q,1}^{\frac{3}{q}}}\|b\|_{\dot{B}_{p,1}^s}.
\end{equation*}
\end{lem}
\begin{proof} By applying Bony's decomposition, we can get that
\begin{equation*}
 ab=T_ab+T_ba+R(a,b).
\end{equation*}
We can obtain by using Lemma \ref{bernstein} that
\begin{equation*}
\begin{split}
\|\Delta_j(T_ab)\|_{L^p}&\lesssim\sum\limits_{|j'-j|\leq5}\|S_{j'-1}a\|_{L^{\infty}}\|\Delta_{j'}b\|_{L^p}\\
&\lesssim\sum\limits_{|j'-j|\leq5}\sum\limits_{j''\leq
j'-2}2^{j''\frac{3}{q}}\|\Delta_{j''}a\|_{L^q}\|\Delta_{j'}b\|_{L^p}\\
&\lesssim d_j2^{-js}\|a\|_{\dot{B}_{q,1}^{\frac{3}{q}}}\|b\|_{\dot{B}_{p,1}^s},
\end{split}
\end{equation*}
and for $1\leq p\leq q$,
$s\leq\frac{3}{q}$, we get by using Lemma \ref{bernstein} once again
that
\begin{equation*}
\|\Delta_j(T_ba)\|_{L^p}\lesssim\sum\limits_{|j'-j|\leq5}\|S_{j'-1}b\|_{L^{\frac{pq}{q-p}}}
\|\Delta_{j'}a\|_{L^q}
\lesssim d_j2^{-js}\|a\|_{\dot{B}_{q,1}^{\frac{3}{q}}}\|b\|_{\dot{B}_{p,1}^s}.
\end{equation*}
Then, in the case where $\frac{1}{p}+\frac{1}{q}>1$, and notice that $s>\frac{3}{p}-3$, we
obtain
\begin{eqnarray*}
\|\Delta_jR(a,b)\|_{L^p}&\lesssim& 2^{j(3-\frac{3}{p})}\sum\limits_{j'\geq
j-N_0}\|\Delta_{j'}a\widetilde{\Delta}_{j'}b\|_{L^1}\\
&\lesssim& 2^{j(3-\frac{3}{p})}\sum\limits_{j'\geq
j-N_0}\|\Delta_{j'}a\|_{L^{p'}}\|\widetilde{\Delta}_{j'}b\|_{L^p}\\
&\lesssim& 2^{j(3-\frac{3}{p})}\sum\limits_{j'\geq
j-N_0}2^{j'(\frac{3}{q}-3+\frac{3}{p})}\|\Delta_{j'}a\|_{L^q}\|\widetilde{\Delta}_{j'}b\|_{L^p}\\
&\lesssim& d_j2^{-js}\|a\|_{\dot{B}_{q,1}^{\frac{3}{q}}}\|b\|_{\dot{B}_{p,1}^s},
\end{eqnarray*}
where $\frac{1}{p}+\frac{1}{p'}=1$.
If $\frac{1}{p}+\frac{1}{q}\leq1$, we use the following estimate
\begin{equation*}
\|\Delta_jR(a,b)\|_{L^p}
\lesssim2^{j\frac{3}{q}}\sum\limits_{j'\geq
j-N_0}\|\Delta_{j'}a\|_{L^q}\|\widetilde{\Delta}_{j'}b\|_{L^p}\\
\lesssim d_j2^{-js}\|a\|_{\dot{B}_{q,1}^{\frac{3}{q}}}\|b\|_{\dot{B}_{p,1}^s}
\end{equation*}
for $s+\frac{3}{q}>0$. So we can deduce the proof of the lemma from
the above estimates.
\end{proof}

 And we also shall frequently use the
following version of Gagliardo-Nirenberg inequality \cite{NL1966}:
for any $1\leq p\leq\infty$,
\begin{equation}\label{G-N inequality}
  \|a\|_{L^{\infty}(\mathbb{R})}\leq
  C\|a\|_{L^p(\mathbb{R})}^{1-\frac{1}{p}}\|\nabla
  a\|_{L^p(\mathbb{R})}^{\frac{1}{p}},
\ \textrm{ when }a\in \mathcal{D}(\mathbb{R}).
\end{equation}
As applications of (\ref{G-N inequality}), the following estimates
shall be used frequently in this article.
\begin{lem}\label{u^3 (1)}
Let $1<p\leq m, r\leq\infty$, $u=(u^h,
u^3)\in\widetilde{L}^{\infty}(\mathbb{R}^+;\dot{B}_{p,1}^{-1+\frac{3}{p}}(\mathbb{R}^3))\cap
L^1(\mathbb{R}^+;\dot{B}_{p,1}^{1+\frac{3}{p}}(\mathbb{R}^3))$ with
$\displaystyle\mathrm{div} u=0$. Then there hold that
\begin{equation}\label{u^3 1}
\|\Delta_ju^3\|_{L_t^2(L_h^m(L_v^r))}\lesssim
d_j2^{-j(\frac{2}{m}+\frac{1}{r})}\|u^3\|_{\widetilde{L}_t^2(\dot{B}_{p,1}^{\frac{3}{p}})}
^{1-\frac{1}{p}+\frac{1}{r}}
\|u^h\|_{\widetilde{L}_t^2(\dot{B}_{p,1}^{\frac{3}{p}})}^{\frac{1}{p}-\frac{1}{r}},
\end{equation}
\begin{equation}\label{u^3 2}
\|\Delta_ju^3\|_{L_t^1(L_h^m(L_v^r))}\lesssim
d_j2^{-j(1+\frac{2}{m}+\frac{1}{r})}\|u^3\|_{L_t^1(\dot{B}_{p,1}^{1+\frac{3}{p}})}
^{1-\frac{1}{p}+\frac{1}{r}}
\|u^h\|_{L_t^1(\dot{B}_{p,1}^{1+\frac{3}{p}})}^{\frac{1}{p}-\frac{1}{r}}.
\end{equation}
\end{lem}
\begin{proof} By applying Lemma \ref{bernstein},
Gagliardo-Nirenberg inequality and
$\partial_3u^3=-\displaystyle\mathrm{div}_hu^h$, we have
\begin{equation*}
\begin{split}
\|\Delta_ju^3\|_{L_t^2(L_h^m(L_v^r))}
&\lesssim2^{-j(\frac{2}{m}-\frac{2}{p})}\|\Delta_ju^3\|_{L_t^2(L_h^p(L_v^r))}\\
&\lesssim2^{-j(\frac{2}{m}-\frac{2}{p})}\|\Delta_ju^3\|_{L_t^2(L^p)}^{1-\frac{1}{p}+\frac{1}{r}}
\|\partial_3(\Delta_ju^3)\|_{L_t^2(L^p)}^{\frac{1}{p}-\frac{1}{r}}\\
&\lesssim
d_j2^{-j(\frac{2}{m}+\frac{1}{r})}\|u^3\|_{\widetilde{L}_t^2(\dot{B}_{p,1}^{\frac{3}{p}})}^{1-\frac{1}{p}+\frac{1}{r}}
\|u^h\|_{\widetilde{L}_t^2(\dot{B}_{p,1}^{\frac{3}{p}})}^{\frac{1}{p}-\frac{1}{r}}.
\end{split}
\end{equation*}
Similarly, we have (\ref{u^3 2}).
\end{proof}

\section{The proof of Theorem \ref{main} for $p\leq q$}\label{section3}
 At first, we give the estimates of the transport equation. We
consider the following free transport equation:
\begin{equation}\label{f}
\partial_t a+u\cdot\nabla a=0, \ \ a|_{t=0}=a_0.
\end{equation}
\begin{prop}\label{trans}
Let $1<p\leq q$, $u=(u^h,u^3)\in
\widetilde{L}_T^\infty(\dot{B}_{p,1}^{-1+\frac{3}{p}}(\mathbb{R}^3))\cap{L}_T^1(\dot{B}_{p,1}^{1+\frac{3}{p}}(\mathbb{R}^3))$
with $\displaystyle\mathrm{div}u=0$ and $a_0 \in
\dot{B}_{q,1}^{\frac{3}{q}}(\mathbb{R}^3)$. Then (\ref{f}) has a unique
solution $a\in \mathcal
{C}([0,T];\dot{B}_{q,1}^{\frac{3}{q}}(\mathbb{R}^3))$ so that
\begin{eqnarray}\label{density}
\|a\|_{\widetilde{L}_t^\infty(\dot{B}_{q,1}^{\frac{3}{q}})}\leq\|a_0\|_{\dot{B}_{q,1}^{\frac{3}{q}}}+
C\|a\|_{\widetilde{L}_t^\infty(\dot{B}_{q,1}^{\frac{3}{q}})}\{\|
u^h\|_{L_t^1({\dot{B}_{p,1}^{1+\frac{3}{p}}})} +\|
u^3\|_{L_t^1({\dot{B}_{p,1}^{1+\frac{3}{p}}})}^{1-\frac{1}{p}} \|
u^h\|_{L_t^1({\dot{B}_{p,1}^{1+\frac{3}{p}}})} ^{\frac{1}{p}}\}
\end{eqnarray}
for any $t\in [0,T].$
\end{prop}
\begin{proof}   The existence and uniqueness of solutions to
(\ref{f}) essentially follow from the estimate (\ref{density}) for
some appropriate solutions to (\ref{f}). For simplicity, here we
just present the estimate (\ref{density}) for smooth enough
solutions of (\ref{f}). In this case, applying Bony's decomposition
(\ref{bony}), we obtain
\begin{equation*}
u\cdot\nabla a = T_u\nabla a + \mathcal {R}(u,\nabla a ).
\end{equation*}
Applying $\Delta_j$ to the above equation and taking $L^2$ inner
product of the resulting equation with $|\Delta_ja|^{q-2}\Delta_ja$
(when $q\in(1, 2)$, we need to make some modification as Proposition
2.1 in \cite{RD01}), we obtain
\begin{eqnarray}\label{f2}
\frac{1}{q}\frac{d}{dt}\|\Delta_ja(t)\|_{L^q}^q+(\Delta_j(T_u\nabla
a)\mid|\Delta_ja|^{q-2}\Delta_ja)+(\Delta_j\mathcal {R}(u,\nabla a
)\mid|\Delta_ja|^{q-2}\Delta_ja)=0.
\end{eqnarray}
And one can get by using a standard commutator's argument that
\begin{eqnarray*}
\begin{split}
(\Delta_j(T_u\nabla a)\mid|\Delta_ja|^{q-2}\Delta_ja) &=\sum\limits_{|j'-j|\leq5}
\{([\Delta_j;S_{j'-1}u]\Delta_{j'}\nabla a\mid|\Delta_ja|^{q-2}\Delta_ja)\\
&\ \ +((S_{j'-1}u-S_{j-1}u)\Delta_j\Delta_{j'}\nabla a\mid|\Delta_ja|^{q-2}\Delta_ja)\}.\\
\end{split}\end{eqnarray*}
Then thanks to (\ref{f2}) and using an argument for the $L^q$ energy
estimate in \cite{RD01}, we arrive at
\begin{equation}\label{f3}
\begin{split}
\|\Delta_ja(t)\|_{L^q} &\leq\|\Delta_ja_0\|_{L^q}
+C\int_0^t\{\sum\limits_{|j'-j|\leq5}(\|[\Delta_j;S_{j'-1}u]\Delta_{j'}\nabla
a(t') \|_{L^q}\\
&\ \ +\|(S_{j'-1}u-S_{j-1}u)\Delta_j\Delta_{j'}\nabla a(t')\|_{L^q})
+\|\Delta_j\mathcal {R}(u,\nabla a )(t')\|_{L^q}\}dt'.
\end{split}\end{equation}
We first get by applying the classical estimate on commutators (see
\cite {RD11}.) and (\ref{u^3 2}) with $m=r=\infty$ that
\begin{equation*}
\begin{split}
& \sum\limits_{|j'-j|\leq5}\|[\Delta_j;S_{j'-1}u]\Delta_{j'}\nabla
a\|_{L_t^1(L^q)}\\
&\ \ \lesssim\sum\limits_{|j'-j|\leq5}(\|S_{j'-1}\nabla
u^h\|_{L_t^1(L^{\infty})}\|\Delta_{j'}a\|_{L_t^{\infty}(L^q)}+\|S_{j'-1}\nabla
u^3\|_{L_t^1(L^{\infty})}\|\Delta_{j'}a\|_{L_t^{\infty}(L^q)})\\
&\ \ \lesssim\sum\limits_{|j'-j|\leq5}\sum\limits_{j''\leq
j'-2}(\|\Delta_{j''}\nabla
u^h\|_{L_t^1(L^{\infty})}\|\Delta_{j'}a\|_{L_t^{\infty}(L^q)}+2^{j''}
\|\Delta_{j''}u^3\|_{L_t^1(L^{\infty})}\|\Delta_{j'}
a\|_{L_t^{\infty}(L^q)})\\
&\ \ \lesssim d_j2^{-j\frac{3}{q}}\|
a\|_{\widetilde{L}_t^{\infty}(\dot{B}_{q,1}^{\frac{3}{q}})}(\|
u^h\|_{L_t^1(\dot{B}_{p,1}^{1+\frac{3}{p}})}+\|
u^3\|_{L_t^1(\dot{B}_{p,1}^{1+\frac{3}{p}})}^{1-\frac{1}{p}}\|
u^h\|_{L_t^1(\dot{B}_{p,1}^{1+\frac{3}{p}})}^{\frac{1}{p}}).
\end{split}
\end{equation*}
Similarly, we get
\begin{equation*}
\begin{split}
&\sum\limits_{|j'-j|\leq5}\|(S_{j'-1}u-S_{j-1}u)\Delta_j\Delta_{j'}\nabla
a\|_{L_t^1(L^q)}\\
&\lesssim \sum\limits_{|j'-j|\leq5}(\|S_{j'-1}\nabla
u^h-S_{j-1}\nabla
u^h\|_{L_t^1(L^{\infty})}\|\Delta_ja\|_{L_t^{\infty}(L^q)}\\
&\ \ +\|S_{j'-1}\nabla
u^3-S_{j-1}\nabla u^3\|_{L_t^1(L^{\infty})}\|\Delta_ja\|_{L_t^{\infty}(L^q)})\\
&\lesssim d_j2^{-j\frac{3}{q}}\|
a\|_{\widetilde{L}_t^{\infty}(\dot{B}_{q,1}^{\frac{3}{q}})}(\|
u^h\|_{L_t^1(\dot{B}_{p,1}^{1+\frac{3}{p}})}+\|
u^3\|_{L_t^1(\dot{B}_{p,1}^{1+\frac{3}{p}})}^{1-\frac{1}{p}}\|
u^h\|_{L_t^1(\dot{B}_{p,1}^{1+\frac{3}{p}})}^{\frac{1}{p}}).
\end{split}
\end{equation*}
For $1<p\leq q$, thanks to Lemma \ref{bernstein} and (\ref{u^3 2})
with $m=q, r=\infty$, we obtain
\begin{equation*}
\begin{split}
& \|\Delta_j\mathcal {R}(u,\nabla a)\|_{L_t^1(L^q)}\\
&\lesssim \sum\limits_{j'\geq
j-N_0}(\|S_{j'+2}\nabla_ha\|_{L_t^{\infty}(L^{\infty})}\|\Delta_{j'}u^h\|_{L_t^1(L^q)}+
\|S_{j'+2}\partial_3a\|_{L_t^{\infty}(L_h^{\infty}(L_v^q))}
\|\Delta_{j'}u^3\|_{L_t^1(L_h^q(L_v^{\infty}))})\\
&\lesssim \sum\limits_{j'\geq j-N_0}\sum\limits_{j''\leq
j'+1}(2^{j''(1+\frac{3}{q})}\|\Delta_{j''}a\|_{L_t^{\infty}(L^q)}
2^{j'(\frac{3}{p}-\frac{3}{q})}
\|\Delta_{j'}u^h\|_{L_t^1(L^p)}\\
&\ \ +2^{j''(1+\frac{2}{q})}\|\Delta_{j''}a\|_{L_t^{\infty}(L^q)}
\|\Delta_{j'}u^3\|_{L_t^1(L_h^q(L_v^{\infty}))})\\
&\lesssim
d_j2^{-j\frac{3}{q}}\|
a\|_{\widetilde{L}_t^{\infty}(\dot{B}_{q,1}^{\frac{3}{q}})}(\|
u^h\|_{L_t^1(\dot{B}_{p,1}^{1+\frac{3}{p}})}+\|
u^3\|_{L_t^1(\dot{B}_{p,1}^{1+\frac{3}{p}})}^{1-\frac{1}{p}}\|
u^h\|_{L_t^1(\dot{B}_{p,1}^{1+\frac{3}{p}})}^{\frac{1}{p}}).
\end{split}
\end{equation*}
Substituting the above estimates into (\ref{f3}) and taking
summation for $j\in \mathbb{Z}$, we conclude the proof of
(\ref{density}).
\end{proof}

 As we all known, deriving the estimate for the pressure term is
the main difficulty in the study of the well-posedness of
incompressible inhomogeneous Navier-Stokes equations. In the
following, our goal is to provide the estimates for the pressure term. We first get taking
$\mathrm{div}$ to the momentum equation of (\ref{equation}) that
\begin{equation}\label{g} \
-\Delta\Pi=\mathrm{div}(a\nabla\Pi)-\mu\mathrm{div}(a\Delta
u)+\sum_{i,j=1}^2\partial_i\partial_j(u^iu^j)
+2\partial_3\mathrm{div}_h(u^3u^h)-2\partial_3(u^3\mathrm{div}_hu^h),
\end{equation}
where, for a vector field $u=(u^h,u^3)$, we denote
$\mathrm{div}_hu^h=\partial_1u^1+\partial_2u^2$.

 The following proposition concerning the estimate of the
pressure will be the main ingredient used in the estimate of $u^h$
and $u^3$. Denote
\begin{equation}
\begin{split}
A(a,u)\stackrel{def}{=}&
\mu\|a\|_{\widetilde{L}_t^\infty(\dot{B}_{q,1}^{\frac{3}{q}})}(\|u^h\|_{L_t^1(\dot{B}_{p,1}^{1+\frac{3}{p}})}+\|u^3\|_{L_t^1(\dot{B}_{p,1}^{1+\frac{3}{p}})})
\\
& +\|
u^3\|_{\widetilde{L}_t^2(\dot{B}_{p,1}^{\frac{3}{p}})}^{1-\frac{1}{p}}\|
u^h\|_{\widetilde{L}_t^2(\dot{B}_{p,1}^{\frac{3}{p}})}^{1+\frac{1}{p}}
+\|u^h\|_{\widetilde{L}_t^\infty(\dot{B}_{p,1}^{-1+\frac{3}{p}})}\|
u^3\|_{L_t^1(\dot{B}_{p,1}^{1+\frac{3}{p}})}^{1-\frac{1}{p}}\|
u^h\|_{L_t^1(\dot{B}_{p,1}^{1+\frac{3}{p}})}^{\frac{1}{p}}\\
&+\| u^h\|_{\widetilde{L}_t^\infty(\dot{B}_{p,1}^{-1+\frac{3}{p}})} \|
u^h\|_{L_t^1(\dot{B}_{p,1}^{1+\frac{3}{p}})}+\|
u^3\|_{\widetilde{L}_t^2(\dot{B}_{p,1}^{\frac{3}{p}})}^{1-\alpha}\|
u^h\|_{\widetilde{L}_t^2(\dot{B}_{p,1}^{\frac{3}{p}})}^{1+\alpha}.
\end{split}
\end{equation}
\begin{prop}\label{pressure}
Let $1<p\leq q<6$ with $\frac{1}{p}-\frac{1}{q}\leq\frac{1}{3}$, $a
\in \widetilde{L}_T^\infty(\dot{B}_{q,1}^{\frac{3}{q}}(\mathbb{R}^3))$ and
$u \in
\widetilde{L}_T^\infty(\dot{B}_{p,1}^{-1+\frac{3}{p}}(\mathbb{R}^3)) \cap
L_T^1(\dot{B}_{p,1}^{1+\frac{3}{p}}(\mathbb{R}^3))$. Then (\ref{g}) has a
unique solution $\nabla\Pi \in
L_T^1(\dot{B}_{p,1}^{-1+\frac{3}{p}}(\mathbb{R}^3))$ which decays to zero
when $|x|\rightarrow\infty$ so that for all $t \in [0,T]$, there
holds
\begin{equation}\label{h1}
\|\nabla\Pi\|_{L_t^1(\dot{B}_{p,1}^{-1+\frac{3}{p}})}
\leq C A(a,u),
\end{equation}
provided that
$C\|a\|_{\widetilde{L}_T^\infty(\dot{B}_{q,1}^{\frac{3}{q}})}\leq\frac{1}{2}$,
and $\alpha$ is defined as (\ref{definition alpha}).
\end{prop}
The proof of this proposition will mainly be based on the following
lemmas :
\begin{lem}\label{pressure1}
Let $p>1$, under the assumptions of Proposition \ref{pressure}, one
has
\begin{equation*}
\begin{split}
\|\Delta_j(u^3u^h)\|_{L_t^1(L^p)}& \lesssim
d_j2^{-j\frac{3}{p}}(\|u^h\|_{\widetilde{L}_t^\infty(\dot{B}_{p,1}^{-1+\frac{3}{p}})}\|
u^3\|_{L_t^1(\dot{B}_{p,1}^{1+\frac{3}{p}})}^{1-\frac{1}{p}}\|
u^h\|_{L_t^1(\dot{B}_{p,1}^{1+\frac{3}{p}})}^{\frac{1}{p}}\\
&\ \ +\|
u^3\|_{\widetilde{L}_t^2(\dot{B}_{p,1}^{\frac{3}{p}})}^{1-\frac{1}{p}}\|
u^h\|_{\widetilde{L}_t^2(\dot{B}_{p,1}^{\frac{3}{p}})}^{1+\frac{1}{p}}).
\end{split}
\end{equation*} for all $t\leq T$.
\end{lem}
\begin{proof} We first get by applying Bony's decomposition
(\ref{bony}) that
\begin{equation}\label{g1}
u^3u^h=T_{u^3}u^h+T_{u^h}u^3+R(u^3,u^h).
\end{equation}
Applying Lemma \ref{bernstein} and (\ref{u^3 1}) with $m=r=\infty$
gives rise to
\begin{equation*}
\begin{split}
\|\Delta_j(T_{u^3}u^h)\|_{L_t^1(L^p)} &
\lesssim\sum\limits_{|j'-j|\leq5}\|S_{j'-1}u^3\|_{L_t^2(L^\infty)}
\|\Delta_{j'}u^h\|_{L_t^2(L^p)}\\
&\ \ \lesssim \sum_{|j'-j|\leq5}\sum_{j''\leq j'-2}
\|\Delta_{j''}u^3\|_{L_t^2(L^{\infty})}\|\Delta_{j'}u^h\|_{L_t^2(L^p)}\\
&\ \ \lesssim
d_j2^{-j\frac{3}{p}}\|u^3\|_{\widetilde{L}_t^2(\dot{B}_{p,1}^{\frac{3}{p}})}^{1-\frac{1}{p}}\|
u^h\|_{\widetilde{L}_t^2(\dot{B}_{p,1}^{\frac{3}{p}})}^{1+\frac{1}{p}}.
\end{split}
\end{equation*}
And while (\ref{u^3 2}) applied with $m=p, r=\infty$ gives
\begin{equation*}
\begin{split} \|\Delta_j(T_{u^h}u^3)\|_{L_t^1(L^p)}
&
\lesssim\sum\limits_{|j'-j|\leq5}\|S_{j'-1}u^h\|_{L_t^\infty(L_h^\infty(L_v^p))}
\|\Delta_{j'}u^3\|_{L_t^1(L_h^p(L_v^\infty))}\\
&\ \ \lesssim\sum\limits_{|j'-j|\leq5}\sum\limits_{j''\leq
j'-2}2^{j''\frac{2}{p}}\|\Delta_{j''}u^h\|_{L_t^\infty(L^p)}\|\Delta_{j'}u^3\|_{L_t^1(L_h^p(L_v^\infty))}\\
&\ \ \lesssim
d_j2^{-j\frac{3}{p}}\|u^h\|_{\widetilde{L}_t^\infty(\dot{B}_{p,1}^{-1+\frac{3}{p}})}\|
u^3\|_{L_t^1(\dot{B}_{p,1}^{1+\frac{3}{p}})}^{1-\frac{1}{p}}\|
u^h\|_{L_t^1(\dot{B}_{p,1}^{1+\frac{3}{p}})}^{\frac{1}{p}},
\end{split}
\end{equation*}
and
\begin{equation*}
\begin{split}
\|\Delta_jR(u^3,u^h)\|_{L_t^1(L^p)}& \lesssim\sum\limits_{j'\geq
j-N_0}\|\Delta_{j'}u^3\|_{L_t^1(L_h^p(L_v^\infty))}\|\widetilde{\Delta}_{j'}u^h\|_{L_t^\infty(L_h^\infty(L_v^p))}\\
&\lesssim\sum\limits_{j'\geq
j-N_0}2^{j'\frac{2}{p}}\|\Delta_{j'}u^3\|_{L_t^1(L_h^p(L_v^\infty))}\|\widetilde{\Delta}_{j'}u^h\|_{L_t^\infty(L^p)}\\
&\lesssim
d_j2^{-j\frac{3}{p}}\|u^h\|_{\widetilde{L}_t^\infty(\dot{B}_{p,1}^{-1+\frac{3}{p}})}\|
u^3\|_{L_t^1(\dot{B}_{p,1}^{1+\frac{3}{p}})}^{1-\frac{1}{p}}\|
u^h\|_{L_t^1(\dot{B}_{p,1}^{1+\frac{3}{p}})}^{\frac{1}{p}}.
\end{split}
\end{equation*}
Along with (\ref{g1}), we prove the inequality of Lemma
\ref{pressure1}. {\hfill $\square$\medskip}
\begin{lem}\label{pressure2}
Under the assumptions of Proposition \ref{pressure}, when $1<p<6$,
one has
\begin{eqnarray*}
\|\Delta_j(u^3\mathrm{div}_hu^h)\|_{L_t^1(L^p)}\lesssim
d_j2^{j(1-\frac{3}{p})}(\|
u^3\|_{\widetilde{L}_t^2(\dot{B}_{p,1}^{\frac{3}{p}})}^{1-\frac{1}{p}}\|
u^h\|_{\widetilde{L}_t^2(\dot{B}_{p,1}^{\frac{3}{p}})}^{1+\frac{1}{p}}+\|
u^3\|_{\widetilde{L}_t^2(\dot{B}_{p,1}^{\frac{3}{p}})}^{1-\alpha}\|
u^h\|_{\widetilde{L}_t^2(\dot{B}_{p,1}^{\frac{3}{p}})}^{1+\alpha}),
\end{eqnarray*}
whereas
\begin{equation*}
 \alpha=\left\{\begin{array}{l}
\frac{1}{p},\ \ 1<p<5\\
\varepsilon,\ \ 5\leq p<6\\
\end{array}\right.
\end{equation*}
for $0<\varepsilon<\frac{6}{p}-1$.
\end{lem}
\textbf{Proof.} We first get by applying Bony's
decomposition(\ref{bony}) that
\begin{equation}\label{g2}
u^3\mathrm{div}_hu^h=T_{u^3}\mathrm{div}_hu^h+T_{\mathrm{div}_hu^h}u^3+R(u^3,\mathrm{div}_hu^h).
\end{equation}
Applying (\ref{u^3 1}) with $m=r=\infty$, we obtain
\begin{equation*}
\begin{split}
\|\Delta_j(T_{u^3}\mathrm{div}_hu^h)\|_{L_t^1(L^p)}
&\lesssim\sum_{|j'-j|\leq5}\sum_{j''\leq j'-2}
\|\Delta_{j''}u^3\|_{L_t^2(L^{\infty})}
\|\Delta_{j'}(\mathrm{div}_hu^h)\|_{L_t^2(L^p)}\\
&\lesssim d_j2^{j(1-\frac{3}{p})}\|
u^3\|_{\widetilde{L}_t^2(\dot{B}_{p,1}^{\frac{3}{p}})}^{1-\frac{1}{p}}\|
u^h\|_{\widetilde{L}_t^2(\dot{B}_{p,1}^{\frac{3}{p}})}^{1+\frac{1}{p}},
\end{split}
\end{equation*}
and similarly, we get
\begin{equation*}
\begin{split}
\|\Delta_j(T_{\mathrm{div}_hu^h}u^3)\|_{L_t^1(L^p)}
&\lesssim\sum\limits_{|j'-j|\leqslant5}\|S_{j'-1}(\mathrm{div}_hu^h)\|_{L_t^2(L_h^\infty(L_v^p))}
\|\Delta_{j'}u^3\|_{L_t^2(L_h^p(L_v^\infty))}\\
&\lesssim\sum\limits_{|j'-j|\leqslant5}\sum\limits_{j''\leqslant
j'-2}2^{j''\frac{2}{p}}\|\Delta_{j''}(\mathrm{div}_hu^h)\|_{L_t^2(L^p)}\|\Delta_{j'}u^3\|_{L_t^2(L_h^p(L_v^\infty))}\\
&\lesssim d_j2^{j(1-\frac{3}{p})}\|
u^3\|_{\widetilde{L}_t^2(\dot{B}_{p,1}^{\frac{3}{p}})}^{1-\frac{1}{p}}\|
u^h\|_{\widetilde{L}_t^2(\dot{B}_{p,1}^{\frac{3}{p}})}^{1+\frac{1}{p}}.
\end{split}
\end{equation*}
For $1<p<5$, while (\ref{u^3 1}) applied with $m=p, r=\infty$ gives
\begin{equation*}
\begin{split}
\|\Delta_j(R(u^3,\mathrm{div}_hu^h))\|_{L_t^1(L^p)}
&\lesssim2^{j\frac{2}{p}}\sum\limits_{j'\geqslant
j-N_0}\|\widetilde{\Delta}_{j'}(\mathrm{div}_hu^h)\Delta_{j'}u^3\|_{L_t^1(L_h^{\frac{p}{2}}(L_v^p))}\\
&\lesssim2^{j\frac{2}{p}}\sum\limits_{j'\geq
j-N_0}\|\widetilde{\Delta}_{j'}(\mathrm{div}_hu^h)\|_{L_t^2(L^p)}\|\Delta_{j'}u^3\|_{L_t^2(L_h^p(L_v^\infty))}\\
&\lesssim d_j2^{j(1-\frac{3}{p})}\|
u^3\|_{\widetilde{L}_t^2(\dot{B}_{p,1}^{\frac{3}{p}})}^{1-\frac{1}{p}}\|
u^h\|_{\widetilde{L}_t^2(\dot{B}_{p,1}^{\frac{3}{p}})}^{1+\frac{1}{p}}.
\end{split}
\end{equation*}
For $5\leq p<6$, we get by using (\ref{u^3 1}) with $m=p,\ \
\frac{1}{r}=\frac{1}{p}-\varepsilon$ that
\begin{equation*}
\begin{split}
\|\Delta_j(R(u^3,\mathrm{div}_hu^h))\|_{L_t^1(L^p)}&\lesssim2^{j(\frac{3}{p}-\varepsilon)}\sum\limits_{j'\geqslant
j-N_0}\|\widetilde{\Delta}_{j'}(\mathrm{div}_hu^h)\Delta_{j'}u^3\|_{L_t^1(L_h^{\frac{p}{2}}(L_v^{1/(\frac{2}{p}-\varepsilon)}))}\\
&\lesssim2^{j(\frac{3}{p}-\varepsilon)}\sum\limits_{j'\geqslant j-N_0}\|\widetilde{\Delta}_{j'}(\mathrm{div}_hu^h)\|_{L_t^2(L^p)}
\|\Delta_{j'}u^3\|_{L_t^2(L_h^p(L_v^{1/(\frac{1}{p}-\varepsilon)}))}\\
&\lesssim d_j2^{j(1-\frac{3}{p})}\|
u^3\|_{\widetilde{L}_t^2(\dot{B}_{p,1}^{\frac{3}{p}})}^{1-\varepsilon}\|
u^h\|_{\widetilde{L}_t^2(\dot{B}_{p,1}^{\frac{3}{p}})}^{1+\varepsilon}
\end{split}
\end{equation*}
where $1-\frac{6}{p}+\varepsilon<0$. Thus again thanks to
(\ref{g2}), we conclude the proof of Lemma \ref{pressure2}.
\end{proof}

\noindent\textbf{Proof of proposition
\ref{pressure}:} Again as both the proof of the existence and
uniqueness of solutions to (\ref{g}) is essentially followed by the
estimates  (\ref{h1}) for some appropriate approximate solutions of
(\ref{g}). For simplicity, we just prove (\ref{h1}) for smooth
enough solutions of (\ref{g}). Indeed thanks to (\ref{g}) and
$\mathrm{div}u=0$, we have
\begin{equation}
\begin{split}\label{g3}
\nabla\Pi&
=\nabla(-\Delta)^{-1}[\mathrm{div}(a\nabla\Pi)+\sum_{i,j=1}^2\partial_i\partial_j(u^iu^j)+2\partial_3\sum\limits_{i=1}^2\partial_i(u^3u^i)\\
&\ \ -2\partial_3(u^3\mathrm{div}_hu^h) -\mu\mathrm{div}_h(a\Delta
u^h)-\mu\partial_3(a\Delta u^3)].
\end{split}
\end{equation}
Applying $\Delta_j$ to the above equation and using Lemma
\ref{bernstein} leads to
\begin{equation}
\begin{split}\label{g4}
\|\Delta_j(\nabla\Pi)\|_{L_t^1(L^p)}&
\lesssim\|\Delta_j(a\nabla\Pi)\|_{L_t^1(L^p)}+2^j(\|\Delta_j(u^h\otimes
u^h)\|_{L_t^1(L^p)}+\|\Delta_j(u^3u^h)\|_{L_t^1(L^p)})\\
&\ \
+\|\Delta_j(u^3\mathrm{div}_hu^h)\|_{L_t^1(L^p)}+\mu\|\Delta_j(a\Delta
u^h)\|_{L_t^1(L^p)}+\mu\|\Delta_j(a\Delta u^3)\|_{L_t^1(L^p)}.
\end{split}
\end{equation}
For $\frac{1}{p}-\frac{1}{q}\leq\frac{1}{3}$, applying Lemma
\ref{product law 2} gives rise to
\begin{eqnarray*}
\|\Delta_j(a\nabla\Pi)\|_{L_t^1(L^p)}&\lesssim&
d_j2^{j(1-\frac{3}{p})}\|a\|_{\widetilde{L}_t^\infty(\dot{B}_{q,1}^{\frac{3}{q}})}\|\nabla\Pi\|_{L_t^1(\dot{B}_{p,1}^{-1+\frac{3}{p}})}\
\
and\\
\|\Delta_j(u^h\otimes u^h)\|_{L_t^1(L^p)}&\lesssim&
d_j2^{-j\frac{3}{p}}\|u^h\|_{\widetilde{L}_t^\infty(\dot{B}_{p,1}^{-1+\frac{3}{p}})}\|u^h\|_{L_t^1(\dot{B}_{p,1}^{1+\frac{3}{p}})},\\
\|\Delta_j(a\Delta u^h)\|_{L_t^1(L^p)}&\lesssim&
d_j2^{j(1-\frac{3}{p})}\|a\|_{\widetilde{L}_t^\infty(\dot{B}_{q,1}^{\frac{3}{q}})}\|u^h\|_{L_t^1(\dot{B}_{p,1}^{1+\frac{3}{p}})},\\
\|\Delta_j(a\Delta u^3)\|_{L_t^1(L^p)}&\lesssim&
d_j2^{j(1-\frac{3}{p})}\|a\|_{\widetilde{L}_t^\infty(\dot{B}_{q,1}^{\frac{3}{q}})}\|u^3\|_{L_t^1(\dot{B}_{p,1}^{1+\frac{3}{p}})},
\end{eqnarray*}
which along with Lemma \ref{pressure1}, Lemma \ref{pressure2} and
(\ref{g4}) implies that
\begin{equation*}
\|\nabla\Pi\|_{L_t^1(\dot{B}_{p,1}^{-1+\frac{3}{p}})}\leq
C\{\|a\|_{\widetilde{L}_t^\infty(\dot{B}_{q,1}^{\frac{3}{q}})}
\|\nabla\Pi\|_{L_t^1(\dot{B}_{p,1}^{-1+\frac{3}{p}})} +A(a,u)\}
\end{equation*}
for all $t\leq T$. So provided that
$C\|a\|_{\widetilde{L}_t^\infty(\dot{B}_{q,1}^{\frac{3}{q}})}\leq\frac{1}{2}$,
 we conclude the proof of (\ref{h1}).    {\hfill $\square$\medskip}\\

 Motivated by
\cite{GP10, PP11, PP12, ZT09}, and based on the estimate of the
pressure, we shall deal with the $L^p$ type energy estimate for
$u^h$ and $u^3$ separately.
\begin{prop}\label{u^h estimate}
Under the assumption of Proposition \ref{pressure} and
\begin{equation}\label{condition 0}
C\|a\|_{\widetilde{L}_T^{\infty}(\dot{B}_{q,1}^{\frac{3}{q}})}\leq\frac{1}{2}
\end{equation} there holds
\begin{equation}\label{horizontal}
\|u^h\|_{\widetilde{L}_t^\infty(\dot{B}_{p,1}^{-1+\frac{3}{p}})}+\bar{c}\mu\|
u^h\|_{L_t^1(\dot{B}_{p,1}^{1+\frac{3}{p}})}\leq \|
u_0^h\|_{\dot{B}_{p,1}^{-1+\frac{3}{p}}}+CA(a,u).
\end{equation}
\end{prop}
\begin{proof} According to the second equation of
(\ref{equation}), we have
\begin{equation*}
\ \partial_tu^h+u\cdot\nabla u^h+(1+a)(\nabla_h\Pi-\mu\Delta u^h)=0.
\end{equation*}
Applying the operator $\Delta_j$ to the above equation and taking
the $L^2$ inner product of the resulting equation with $|\Delta_j
u^h|^{p-2}\Delta_j u^h$ (when $p\in(1,2)$, we need to make some
modification as Proposition 2.1 in \cite{RD01}), we obtain
\begin{equation}\label{aa}
\begin{split}
\frac{1}{p}\frac{d}{dt}\|\Delta_j
u^h\|_{L^p}^p-\mu\int\limits_{\mathbb{R}^3}\Delta\Delta_j
u^h|\Delta_j u^h|^{p-2}\Delta_j u^h\,\mathrm{d}x
=-\int\limits_{\mathbb{R}^3}\{\Delta_j (u\cdot\nabla
u^h)\\
+\Delta_j((1+a)\nabla_h \Pi) -\mu\Delta_j(a\Delta u^h)\}|\Delta_j
u^h|^{p-2}\Delta_j u^h\,\mathrm{d}x.
\end{split}
\end{equation}
However thanks to  Lemma A.5 of Appendix in \cite{RD01}, there
exists a positive constant $\bar{c}$ so that
\begin{equation*}
-\int\limits_{\mathbb{R}^3}\Delta\Delta_j u^h|\Delta_j
u^h|^{p-2}\Delta_j u^h\,\mathrm{d}x\geq \bar{c} 2^{2j}\|\Delta_j
u^h\|_{L^p}^{p} ,
\end{equation*}
so we get from (\ref{aa}) that
\begin{equation}\label{h3}
\begin{split}
& \frac{d}{dt}\|\Delta_j u^h\|_{L^p}+\bar{c}\mu2^{2j}\|\Delta_j
u^h\|_{L^p}\\
&\ \ \leq\|\Delta_j(u\cdot\nabla u^h)\|_{L^p}
+\|\Delta_j((1+a)\nabla_h\Pi)\|_{L^p}+\mu\|\Delta_j(a\Delta
u^h)\|_{L^p},
\end{split}
\end{equation}
and integrating in time, we get
\begin{equation}
\begin{split}\label{h4}
\|\Delta_j u^h\|_{L^\infty_t(L^p)}+\bar{c}\mu2^{2j}\|\Delta_j
u^h\|_{L_t^1(L^p)} &\leq\|\Delta_j
u^h_0\|_{L^p}+\|\Delta_j(u\cdot\nabla u^h)\|_{L_t^1(L^p)}\\
 &\ \
+\|\Delta_j((1+a)\nabla_h\Pi)\|_{L_t^1(L^p)} +\mu\|\Delta_j(a\Delta
u^h)\|_{L_t^1(L^p)}.
\end{split}
\end{equation}
Applying Lemma \ref{product law 2} and Lemma \ref{pressure1}, we
obtain
\begin{equation*}
\begin{split}
\|\Delta_j(u\cdot\nabla u^h)\|_{L_t^1(L^p)}&
\leq\|\Delta_j\mathrm{div}_h(u^h\otimes
u^h)\|_{L_t^1(L^p)}+\|\Delta_j(\partial_3(u^3u^h))\|_{L_t^1(L^p)}\\
&\lesssim2^j(\|\Delta_j(u^h\otimes
u^h)\|_{L_t^1(L^p)}+\|\Delta_j(u^3u^h)\|_{L_t^1(L^p)})\\
&\lesssim d_j2^{j(1-\frac{3}{p})}(\|
u^h\|_{\widetilde{L}_t^\infty(\dot{B}_{p,1}^{-1+\frac{3}{p}})} \|
u^h\|_{L_t^1(\dot{B}_{p,1}^{1+\frac{3}{p}})}+\|
u^3\|_{\widetilde{L}_t^2(\dot{B}_{p,1}^{\frac{3}{p}})}^{1-\frac{1}{p}}\|
u^h\|_{\widetilde{L}_t^2(\dot{B}_{p,1}^{\frac{3}{p}})}^{1+\frac{1}{p}}\\
&\ \ +\| u^h \|_{\widetilde{L}_t^\infty(\dot{B}_{p,1}^{-1+\frac{3}{p}})}\|
u^3\|_{L_t^1(\dot{B}_{p,1}^{1+\frac{3}{p}})}^{1-\frac{1}{p}}\|
u^h\|_{L_t^1(\dot{B}_{p,1}^{1+\frac{3}{p}})}^{\frac{1}{p}}).
\end{split}
\end{equation*}
While applying Lemma \ref{product law 2} and Proposition
\ref{pressure},  under the assumption (\ref{condition 0}), we arrive at
\begin{equation*}
\begin{split}
\|\Delta_j((1+a)\nabla_h\Pi)\|_{L_t^1(L^p)}&\lesssim
d_j2^{j(1-\frac{3}{p})}(1+\| a\|_{
\widetilde{L}_t^\infty(\dot{B}_{q,1}^{\frac{3}{q}})})\|\nabla_h\Pi\|_
{L_t^1(\dot{B}_{p,1}^{-1+\frac{3}{p}})}\\
&\lesssim d_j2^{j(1-\frac{3}{p})}A(a,u)
\end{split}
\end{equation*}
and
\begin{equation*}
\|\Delta_j(a\Delta u^h)\|_{L_t^1(L^p)}\lesssim
d_j2^{j(1-\frac{3}{p})} \|
a\|_{\widetilde{L}_t^\infty(\dot{B}_{p,1}^{\frac{3}{p}})}\|
u^h\|_{L_t^1(\dot{B}_{p,1}^{1+\frac{3}{p}})}.
\end{equation*}
Substituting the above estimates into (\ref{h4}) and the condition
(\ref{condition 0}), we deduce the proof of (\ref{horizontal}).
\end{proof}

\begin{prop}\label{u^3 estimate}
Under the assumption of Proposition \ref{pressure}, we have
\begin{equation}\label{vertical}
\| u^3\|_{\widetilde{L}_t^\infty(\dot{B}_{p,1}^{-1+\frac{3}{p}})}
+\bar{c}\mu\| u^3\|_{L_t^1(\dot{B}_{p,1}^{1+\frac{3}{p}})}
\leq\|
u_0^3\|_{\dot{B}_{p,1}^{-1+\frac{3}{p}}}+C\{\|u^3\|_{\widetilde{L}^2_t(\dot{B}^{\frac{3}{p}}_{p,1})}\|u^h\|_{\widetilde{L}^2_t(\dot{B}^{\frac{3}{p}}_{p,1})}+A(a,u)\}.
\end{equation}
\end{prop}
\begin{proof} According to the second equation of
(\ref{equation}), we have
\begin{equation*}
\partial_tu^3+u\cdot\nabla u^3+(1+a)(\partial_3\Pi-\mu\Delta u^3)=0,
\end{equation*}
we get by a similar derivation of (\ref{h3}) that
\begin{equation}
\begin{split}\label{h5}
\|\Delta_ju^3(t)\|_{L^p}+\bar{c}\mu2^{2j}\|\Delta_j
u^3\|_{L_t^1(L^p)}& \leq\|\Delta_ju_0^3\|_{L^p}+\|
\Delta_j(u\cdot\nabla
u^3)\|_{L_t^1(L^p)}\\
&\ \ +\|\Delta_j((1+a)\partial_3\Pi)\|_{L_t^1(L^p)}+
\mu\|\Delta_j(a\Delta u^3)\|_{L_t^1(L^p)}.
\end{split}
\end{equation}
By using Lemma \ref{bernstein} and Lemma \ref{product law 2}, we
deduce that
\begin{equation*}
\begin{split}
\|\Delta_j(u\cdot\nabla
u^3)\|_{L_t^1(L^p)}&\lesssim2^j\|\Delta_j(u^hu^3)\|_{L_t^1(L^p)}+\|\Delta_j(u^3
\mathrm{div}_hu^h)\|_{L_t^1(L^p)}\\
&\lesssim
d_j2^{j(1-\frac{3}{p})}\|u^3\|_{\widetilde{L}^2_t(\dot{B}^{\frac{3}{p}}_{p,1})}
\|u^h\|_{\widetilde{L}^2_t(\dot{B}^{\frac{3}{p}}_{p,1})}.
\end{split}
\end{equation*}
We get by using Proposition \ref{pressure} and (\ref{condition 0})
that
\begin{equation*}
\|\Delta_j((1+a)\partial_3\Pi)\|_{L_t^1(L^p)}
\lesssim d_j2^{j(1-\frac{3}{p})}A(a,u)
\end{equation*} and
\begin{equation*}
\|\Delta_j(a\Delta u^3)\|_{L_t^1(L^p)}\lesssim
 d_j2^{j(1-\frac{3}{p})}\|
a\|_{\widetilde{L}_t^\infty(\dot{B}_{q,1}^{\frac{3}{q}})}\|
u^3\|_{L_t^1(\dot{B}_{p,1}^{1+\frac{3}{p}})}.
\end{equation*}
Then we obtain (\ref{vertical}) by substituting the above estimates
into (\ref{h5}).
\end{proof}

\noindent \textbf{The proof of Theorem \ref{main} for $p\leq q$}: Indeed given $a_0 \in
\dot{B}_{q,1}^{\frac{3}{q}}(\mathbb{R}^3)$ and $u_0 \in
\dot{B}_{p,1}^{-1+\frac{3}{p}}(\mathbb{R}^3)$ with
$\|a_0\|_{\dot{B}_{q,1}^{\frac{3}{q}}}$ sufficiently small and $p, q$
satisfying the conditions listed in Theorem \ref{main} for $p\leq
q$, Theorem \ref{abidi theorem}
ensures that there exists a positive time $T$ so that the system
(\ref{equation}) has a unique solution $(a,u,\nabla\Pi)\in E_{p,q,T}$.
We denote $T^*$ to be the largest existence time. Hence to prove Theorem \ref{main} for $p\leq q$, we only
need to prove that $T^*=\infty$.

 Now let $\eta$ be a small
enough positive constant, which will
 be determined later on. We define $\mathcal {T}$ by
\begin{equation}
\begin{split}\label{continuity assumption}
\mathcal {T}\stackrel{def}{=}\sup\{t\in [0,T^*):
\|a\|_{\widetilde{L}_t^{\infty}(\dot{B}_{q,1}^{\frac{3}{q}})}\leq
4\|a_0\|_{\dot{B}_{q,1}^{\frac{3}{q}}},\\
\| u^3\|_{\widetilde{L}_t^\infty(\dot{B}_{p,1}^{-1+\frac{3}{p}})}
+\bar{c}\mu\| u^3\|_{L_t^1(\dot{B}_{p,1}^{1+\frac{3}{p}})}\leq
4\|u_0^3\|_{\dot{B}_{p,1}^{-1+\frac{3}{p}}}+\mu,\\
\| u^h\|_{\widetilde{L}_t^\infty(\dot{B}_{p,1}^{-1+\frac{3}{p}})}
+\bar{c}\mu\| u^h\|_{L_t^1(\dot{B}_{p,1}^{1+\frac{3}{p}})}\leq
4\|u_0^h\|_{\dot{B}_{p,1}^{-1+\frac{3}{p}}}+\eta\mu\}.
\end{split}
\end{equation}
In what follows, we shall prove that $\mathcal {T}=T^*$ under the
assumptions of (\ref{data condition0})-(\ref{data condition1}).

If not, we assume that
$$\mathcal {T}<T^*.$$
And for $t\leq \mathcal {T}$,
we deduce from (\ref{density}) that
\begin{equation*}
\begin{split}
\|a\|_{\widetilde{L}_t^{\infty}(\dot{B}_{q,1}^{\frac{3}{q}})}&\leq
\|a_0\|_{\dot{B}_{q,1}^{\frac{3}{q}}}+\frac{C}{\bar{c}}
\|a\|_{\widetilde{L}_t^{\infty}(\dot{B}_{q,1}^{\frac{3}{q}})}\{\frac{4}{\mu}
\|u_0^h\|_{\dot{B}_{p,1}^{-1+\frac{3}{p}}}+\eta\\
&\ \
+(\frac{4}{\mu}\|u_0^3\|_{\dot{B}_{p,1}^{-1+\frac{3}{p}}}+1)^{1-\frac{1}{p}}
(\frac{4}{\mu}\|u_0^h\|_{\dot{B}_{p,1}^{-1+\frac{3}{p}}}+\eta)^{\frac{1}{p}}\}.
\end{split}
\end{equation*}
By taking
\begin{equation}\label{condition 1}
\frac{C}{\bar{c}}\{\frac{4}{\mu}
\|u_0^h\|_{\dot{B}_{p,1}^{-1+\frac{3}{p}}}+\eta+(\frac{4}{\mu}\|u_0^3\|_{\dot{B}_{p,1}^{-1+\frac{3}{p}}}+1)^{1-\frac{1}{p}}
(\frac{4}{\mu}\|u_0^h\|_{\dot{B}_{p,1}^{-1+\frac{3}{p}}}+\eta)^{\frac{1}{p}}\}<\frac{1}{2},
\end{equation}
we obtain
\begin{equation}\label{result 1}
\|a\|_{\widetilde{L}_t^{\infty}(\dot{B}_{q,1}^{\frac{3}{q}})}\leq
2\|a_0\|_{\dot{B}_{q,1}^{\frac{3}{q}}}
\end{equation}
for $t\leq\mathcal {T}$. Thanks to (\ref{vertical}), on the one
hand, we obtain
\begin{equation*}
\begin{split}
&\| u^3\|_{\widetilde{L}_t^{\infty}(\dot{B}_{p,1}^{-1+\frac{3}{p}})}
+\bar{c}\mu\|u^3\|_{L_t^1(\dot{B}_{p,1}^{1+\frac{3}{p}})}\\
&\leq\|
u_0^3\|_{\dot{B}_{p,1}^{-1+\frac{3}{p}}}+\frac{C}{\bar{c}\mu}(\|
u^3\|_{\widetilde{L}_t^{\infty}(\dot{B}_{p,1}^{-1+\frac{3}{p}})}+\bar{c}\mu\|u^3\|_{L_t^1(\dot{B}_{p,1}^{1+\frac{3}{p}})})
\\
&\ \
\times(4\|u_0^h\|_{\dot{B}_{p,1}^{-1+\frac{3}{p}}}+\eta\mu)
+C\|a_0\|_{\dot{B}_{q,1}^{\frac{3}{q}}}(4\|u_0^h\|_{\dot{B}_{p,1}^{-1+\frac{3}{p}}}+\eta\mu)
+C\mu\|a_0\|_{\dot{B}_{q,1}^{\frac{3}{q}}}\|u^3\|_{L_t^1(\dot{B}_{p,1}^{1+\frac{3}{p}})}
\\
&\ \
+\frac{C}{\bar{c}\mu}(\|u_0^h\|_{\dot{B}_{p,1}^{-1+\frac{3}{p}}}+\eta\mu)^{1+\frac{1}{p}}
(\|u_0^3\|_{\dot{B}_{p,1}^{-1+\frac{3}{p}}}+\mu)^{1-\frac{1}{p}}+\frac{C}{\bar{c}\mu}
(\|u_0^h\|_{\dot{B}_{p,1}^{-1+\frac{3}{p}}}+\eta\mu)^2
\\
&\ \
+\frac{C}{\bar{c}\mu}(\|u_0^h\|_{\dot{B}_{p,1}^{-1+\frac{3}{p}}}+\eta\mu)^{1+\alpha}
(\|u_0^3\|_{\dot{B}_{p,1}^{-1+\frac{3}{p}}}+\mu)^{1-\alpha}
\end{split}
\end{equation*}
for $t\leq\mathcal {T}$, and taking
\begin{equation}\label{condition 2}
\begin{split}
&\frac{C}{\bar{c}\mu}\{(\|u_0^h\|_{\dot{B}_{p,1}^{-1+\frac{3}{p}}}+\eta\mu)^{\frac{1}{p}}
(\|u_0^3\|_{\dot{B}_{p,1}^{-1+\frac{3}{p}}}+\mu)^{1-\frac{1}{p}}+(\|u_0^h\|_{\dot{B}_{p,1}^{-1+\frac{3}{p}}}+\eta\mu)^{\alpha}
\\
&\ \
\times(\|u_0^3\|_{\dot{B}_{p,1}^{-1+\frac{3}{p}}}+\mu)^{1-\alpha}+(4\|u_0^h\|_{\dot{B}_{p,1}^{-1+\frac{3}{p}}}+\eta\mu)\}\leq\frac{1}{8}
\end{split}
\end{equation}
and combing with
\begin{equation}\label{condition5}
C\|a_0\|_{\dot{B}_{q,1}^{\frac{3}{q}}}\leq\frac{1}{8}\bar{c},
\end{equation}
we obtain
$$
\| u^3\|_{\widetilde{L}_t^{\infty}(\dot{B}_{p,1}^{-1+\frac{3}{p}})}
+\bar{c}\mu\|u^3\|_{L_t^1(\dot{B}_{p,1}^{1+\frac{3}{p}})}\leq2\|
u_0^3\|_{\dot{B}_{p,1}^{-1+\frac{3}{p}}}+2(\|u_0^h\|_{\dot{B}_{p,1}^{-1+\frac{3}{p}}}+\eta\mu).
$$
While taking
\begin{equation}\label{condition 3}
\|u_0^h\|_{\dot{B}_{p,1}^{-1+\frac{3}{p}}}+\eta\mu\leq\frac{\mu}{4},
\end{equation}
we get that
\begin{equation}\label{result 2}
 \|
u^3\|_{\widetilde{L}_t^{\infty}(\dot{B}_{p,1}^{-1+\frac{3}{p}})}
+\bar{c}\mu\|u^3\|_{L_t^1(\dot{B}_{p,1}^{1+\frac{3}{p}})}\leq2\|
u_0^3\|_{\dot{B}_{p,1}^{-1+\frac{3}{p}}}+\frac{\mu}{2}.
\end{equation}
On the other hand, for $t\leq\mathcal {T}$, we can deduce from
(\ref{horizontal}) that
\begin{equation*}
\begin{split}
&\|
u^h\|_{\widetilde{L}_t^\infty(\dot{B}_{p,1}^{-1+\frac{3}{p}})}+\bar{c}\mu\|
u^h\|_{L_t^1(\dot{B}_{p,1}^{1+\frac{3}{p}})}\leq
\|u_0^h\|_{\dot{B}_{p,1}^{-1+\frac{3}{p}}}
+C(\|u_0^h\|_{\dot{B}_{p,1}^{-1+\frac{3}{p}}}+\eta\mu)\|
u^h\|_{L_t^1(\dot{B}_{p,1}^{1+\frac{3}{p}})}\\
&\ \ +\frac{C}{\bar{c}\mu}\|
u^h\|_{\widetilde{L}_t^\infty(\dot{B}_{p,1}^{-1+\frac{3}{p}})}(\|u_0^h\|_{\dot{B}_{p,1}^{-1+\frac{3}{p}}}+\eta\mu)^{\frac{1}{p}}
(\|u_0^3\|_{\dot{B}_{p,1}^{-1+\frac{3}{p}}}+\mu)^{1-\frac{1}{p}}\\
&\ \ +C\mu\|a_0\|_{\dot{B}_{q,1}^{\frac{3}{q}}} \|
u^h\|_{L_t^1(\dot{B}_{p,1}^{1+\frac{3}{p}})}
+\frac{C}{\bar{c}}\|a_0\|_{\dot{B}_{q,1}^{\frac{3}{q}}}(\|u_0^3\|_{\dot{B}_{p,1}^{-1+\frac{3}{p}}}+\mu)
\\
&\ \ +\frac{C}{\bar{c}\mu}\|
u^h\|_{\widetilde{L}_t^\infty(\dot{B}_{p,1}^{-1+\frac{3}{p}})}(\|u_0^h\|_{\dot{B}_{p,1}^{-1+\frac{3}{p}}}+\eta\mu)^{\alpha}
(\|u_0^3\|_{\dot{B}_{p,1}^{-1+\frac{3}{p}}}+\mu)^{1-\alpha}.
\end{split}
\end{equation*}
By taking the conditions (\ref{condition 2})-(\ref{condition5}),
chossing
$\eta=\frac{8C}{\bar{c}}\|a_0\|_{\dot{B}_{q,1}^{\frac{3}{q}}}
(\|u_0^3\|_{\dot{B}_{p,1}^{-1+\frac{3}{p}}}/{\mu}+1)$,
 we obtain
\begin{equation}\label{result 3}
\|
u^h\|_{\widetilde{L}_t^\infty(\dot{B}_{p,1}^{-1+\frac{3}{p}})}+\bar{c}\mu\|
u^h\|_{L_t^1(\dot{B}_{p,1}^{1+\frac{3}{p}})}\leq2\|u_0^h\|_{\dot{B}_{p,1}^{-1+\frac{3}{p}}}+\frac{\eta\mu}{4}
\end{equation}
for $t\leq\mathcal {T}$. Combining (\ref{condition 0}),
(\ref{condition 1}), (\ref{condition 2}), (\ref{condition5}) and
(\ref{condition 3}), we can reach
(\ref{result 1}), (\ref{result 2}) and (\ref{result 3}) if we take
$C$ large enough in (\ref{data condition0})-(\ref{data condition1}). And this contradicts
with the definition (\ref{continuity assumption}), thus we conclude
that $\mathcal {T}=T^*$. Then we complete the proof of Theorem
 \ref{main} for $p\leq q$ by standard continuation argument. {\hfill $\square$\medskip}

\section{The proof of Theorem \ref{main} for $p>q$} \label{p>q}
 For $1<q<
p<6$ with $\frac{1}{q}-\frac{1}{p}\leq\frac{1}{3}$, $a_0\in
\dot{B}_{q,1}^{\frac{3}{q}}(\mathbb{R}^3)$, $u_0\in
\dot{B}_{p,1}^{-1+\frac{3}{p}}(\mathbb{R}^3)$, by using the embedding of
Besov spaces , we get that $a_0\in
\dot{B}_{p,1}^{\frac{3}{p}}(\mathbb{R}^3)$. So by the proof of Theorem
\ref{main} for $p=q$ in Section \ref{section3}, there exists a unique global solution
$(a,u\,\nabla\Pi)\in E_{p,p}$.
Then, by using the same method of Proposition \ref{trans}, we
imply $a\in
\widetilde{L}^{\infty}(\mathbb{R}^+;\dot{B}_{q,1}^{\frac{3}{q}}(\mathbb{R}^3))$
in the following.

 By applying Bony's decomposition (\ref{bony}), we
get
\begin{equation}\label{r3}
\begin{split}
\|\Delta_ja(t)\|_{L^q} &\leq\|\Delta_ja_0\|_{L^q}
+C\int_0^t\{\sum\limits_{|j'-j|\leq5}(\|[\Delta_j;S_{j'-1}u]\Delta_{j'}\nabla
a\|_{L^q}\\
&\ \ +\|(S_{j'-1}u-S_{j-1}u)\Delta_j\Delta_{j'}\nabla a\|_{L^q})
+\|T_{\nabla a}u\|_{L^q}+\|R(u,\nabla a )\|_{L^q}\}dt'.
\end{split}
\end{equation}
We get by applying Lemma \ref{bernstein} and the classical estimate
on commutators that
\begin{equation*}
 \sum\limits_{|j'-j|\leq5}\|[\Delta_j;S_{j'-1}u]\Delta_{j'}\nabla
a\|_{L^q}\lesssim\sum\limits_{|j'-j|\leq5}\|S_{j'-1}\nabla
u\|_{L^{\infty}}\|\Delta_{j'}a\|_{L^q}\lesssim
d_j2^{-j\frac{3}{q}}\|a\|_{\dot{B}_{q,1}^{\frac{3}{q}}}\|u\|_{\dot{B}_{p,1}^{1+\frac{3}{p}}}.
\end{equation*}
Similarly, we obtain
\begin{equation*}
\begin{split}
\sum\limits_{|j'-j|\leq5}\|(S_{j'-1}u-S_{j-1}u)\Delta_j\Delta_{j'}\nabla
a\|_{L^q}&\lesssim\sum\limits_{|j'-j|\leq5}\|S_{j'-1}\nabla
u-S_{j-1}\nabla u\|_{L^{\infty}}\|\Delta_{j'}a\|_{L^q}\\
&\lesssim
d_j2^{-j\frac{3}{q}}\|a\|_{\dot{B}_{q,1}^{\frac{3}{q}}}\|u\|_{\dot{B}_{p,1}^{1+\frac{3}{p}}}.
\end{split}
\end{equation*}
Because of $q< p$ and $\frac{1}{q}-\frac{1}{p}\leq\frac{1}{3}$,
we obtain
\begin{equation*}
\begin{split}
\|T_{\nabla
a}u\|_{L^q}&\lesssim\sum\limits_{|j'-j|\leq5}\|S_{j'-1}\nabla
a\|_{L^{\frac{pq}{p-q}}}\|\Delta_{j'}u\|_{L^p}\\
&\lesssim\sum\limits_{|j'-j|\leq5}\sum\limits_{j''\leq
j'-2}2^{j''(1+\frac{3}{p})}\|\Delta_{j''}a\|_{L^q}\|\Delta_{j'}u\|_{L^p}\\
&\lesssim
d_j2^{-j\frac{3}{q}}\|a\|_{\dot{B}_{q,1}^{\frac{3}{q}}}\|u\|_{\dot{B}_{p,1}^{1+\frac{3}{p}}}.
\end{split}
\end{equation*}
Finally, thanks to $\frac{1}{q}-\frac{1}{p}\leq \frac{1}{3}$, we
arrive at
\begin{equation*}
\begin{split}
\|R(u,\nabla a)\|_{L^q}&\lesssim\sum\limits_{j'\geq
j-N_0}\|\Delta_{j'}u\|_{L^p}\|\widetilde{\Delta}_{j'}\nabla
a\|_{L^{\frac{pq}{p-q}}}\\
&\lesssim\sum\limits_{j'\geq
j-N_0}2^{j'(1+\frac{3}{p})}\|\Delta_{j'}u\|_{L^p}\|\widetilde{\Delta}_{j'}a\|_{L^q}\\
&\lesssim
d_j2^{-j\frac{3}{q}}\|a\|_{\dot{B}_{q,1}^{\frac{3}{q}}}\|u\|_{\dot{B}_{p,1}^{1+\frac{3}{p}}}.
\end{split}
\end{equation*}
Substituting the above estimates into (\ref{r3}), we can deduce that
\begin{equation*}
 \|a\|_{\widetilde{L}^{\infty}_t(\dot{B}_{q,1}^{\frac{3}{q}})}\leq\|a_0\|_{\dot{B}_{q,1}^{\frac{3}{q}}}+C\int_0^t\|a(t')\|_{\dot{B}_{q,1}^{\frac{3}{q}}}
\|u(t')\|_{\dot{B}_{p,1}^{1+\frac{3}{p}}}dt'.
\end{equation*}
Applying Gronwall's inequality, we get that
\begin{equation*}
\|a(t)\|_{\widetilde{L}^{\infty}_t(\dot{B}_{q,1}^{\frac{3}{q}})}
\leq\|a_0\|_{\dot{B}_{q,1}^{\frac{3}{q}}}
\exp\{C\|u\|_{L^1([0,\infty);\dot{B}_{p,1}^{1+\frac{3}{p}})}\}
\end{equation*}
for any $t>0$. Therefore we obtain $a\in
\widetilde{L}^{\infty}(\mathbb{R}^+;\dot{B}_{q,1}^{\frac{3}{q}}(\mathbb{R}^3))$.
{\hfill $\square$\medskip}
$$
\textbf{Appendix}
$$

 In this section, we shall give the proof of Theorem \ref{main
theorem} briefly. At first, for the convenience of the readers, we
recall the definition of multiplier spaces to Besov spaces from
\cite{VG09} and some facts in \cite{RD11}:
\begin{defn}[see Chapter 4 in \cite{VG09}]\label{multiplier space definition}
We call $f$ belonging to the multiplier space,
$\mathscr{M}(\dot{B}_{p,1}^s(\mathbb{R}^n))$, of $\dot{B}_{p,1}^s(\mathbb{R}^n)$
if the distributions $f$ satisfies $\Psi f\in
\dot{B}_{p,1}^s(\mathbb{R}^n)$ whenever $\Psi \in
\dot{B}_{p,1}^s(\mathbb{R}^n)$. We endow this space with the norm
$$
\|f\|_{\mathscr{M}(\dot{B}_{p,1}^s)}\stackrel{def}{=}
\sup\limits_{\|\Psi\|_{\dot{B}_{p,1}^s}=1}\|\Psi f\|_{\dot{B}_{p,1}^s}\ \ for \
\ f\in \mathscr{M}(\dot{B}_{p,1}^s(\mathbb{R}^n)).
$$
\end{defn}

 The estimate of transport equation basically follows from Sect.2 of \cite{HPP12}.
Indeed as we shall not use Lagrange approach as that in \cite{DM12},
we need first to investigate the following transport equation:
\begin{equation}\label{tranport equation}
  \partial_ta+u\cdot\nabla a=0,\ \ a|_{t=0}=a_0,
\end{equation}
with the initial datum $a_0\in\mathcal{M}(\dot{B}_{p,1}^s(\mathbb{R}^n))$. We denoted $X_u(t,y)$ to be the flow map
determined by $u$, namely,
\begin{equation}\label{flow}
X_u(t,y)=y+\int_0^tu(\tau,X_u(\tau,y))d\tau.
\end{equation}
\begin{lem}[Lemma 3.1 in \cite{HPP12}]\label{flow estimate Lemma} Let $s\in (-1,1)$, $p\geq1$, $a\in
\dot{B}_{p,1}^s(\mathbb{R}^n)$, $u\in L^1((0,T);Lip(\mathbb{R}^n))$, and
$X_u$ the flow map determined by (\ref{flow}). Then $a\circ X_u\in
L^{\infty}((0,T);\dot{B}_{p,1}^s(\mathbb{R}^n))$, and there holds
\begin{equation}\label{flow estimate}
\|a\circ X_u\|_{L^{\infty}_t(\dot{B}_{p,1}^s)}\leq
C\|a\|_{\dot{B}_{p,1}^s}\exp\{C\int_0^t\|\nabla
u(\tau)\|_{L^{\infty}}d\tau\}.
\end{equation}
\end{lem}
And we obtain the  following estimate of
$\|a\|_{L_t^{\infty}(\mathscr{M}(\dot{B}_{p,1}^s))}$ by applying the above
Lemma.
\begin{prop}
Let $s\in(-1,1)$ and $p\geq1$. Let $u\in L_T^1(Lip(\mathbb{R}^3))$
and $a_0\in \mathscr{M}(\dot{B}_{p,1}^s(\mathbb{R}^3))$.Then
(\ref{tranport equation}) has a unique solution $a\in
L^{\infty}([0,T];\mathscr{M}(\dot{B}_{p,1}^s(\mathbb{R}^3)))$ so that
\begin{equation}\label{transport equation eatimate}
\|a\|_{L_t^{\infty}(\mathscr{M}(\dot{B}_{p,1}^s))}\leq
\|a_0\|_{\mathscr{M}(\dot{B}_{p,1}^s)}\exp \{C(\|
u^h\|_{L_t^1(\dot{B}_{p,1}^{1+\frac{3}{p}})} +\|
u^3\|_{L_t^1(\dot{B}_{p,1}^{1+\frac{3}{p}})}^{1-\frac{1}{p}}\|
u^h\|_{L_t^1(\dot{B}_{p,1}^{1+\frac{3}{p}})}^{\frac{1}{p}})\}
\end{equation}
for any $t\in [0,T]$.
\end{prop}
\begin{proof} Thanks to (\ref{flow}), we deduce from
(\ref{tranport equation}) that $a(t,x)=a_0(X_u^{-1}(t,x))$. Then
thanks to Definition \ref{multiplier space definition} and Lemma
\ref{flow estimate Lemma}, we obtain
\begin{equation*}
\begin{split}
\|a(t)\|_{\mathscr{M}(\dot{B}_{p,1}^s)}&\leq
\sup\limits_{\|\Psi\|_{\dot{B}_{p,1}^s}=1}\|\Psi a(t)\|_{\dot{B}_{p,1}^s}
\\
&\leq \sup\limits_{\|\Psi\|_{\dot{B}_{p,1}^s}=1}\|(\Psi\circ
X_u(t)a_0)\circ X_u^{-1}(t)\|_{\dot{B}_{p,1}^s}\\
&\leq C\sup\limits_{\|\Psi\|_{\dot{B}_{p,1}^s}=1}\|(\Psi\circ
X_u(t))a_0\|_{\dot{B}_{p,1}^s}\exp\{C\int_0^t\|\nabla
u(\tau)\|_{L^{\infty}}d\tau\}\\
&\leq
C\|a_0\|_{\mathscr{M}(\dot{B}_{p,1}^s)}\exp\{C\int_0^t\|\nabla
u(\tau)\|_{L^{\infty}}d\tau\}\sup\limits_{\|\Psi\|_{\dot{B}_{p,1}^s}=1}\|\Psi\circ
X_u(t)\|_{\dot{B}_{p,1}^s}\\
&\leq C\|a_0\|_{\mathscr{M}(\dot{B}_{p,1}^s)}\exp\{C\int_0^t\|\nabla
u(\tau)\|_{L^{\infty}}d\tau\}.
\end{split}
\end{equation*}
Then we obtain
\begin{equation}\label{transport mid eatimate}
\begin{split}
\|a(t)\|_{\mathscr{M}(\dot{B}_{p,1}^s)}&\leq
C\|a_0\|_{\mathscr{M}(\dot{B}_{p,1}^s)}\exp\{C(\int_0^t\|\nabla
u^h(\tau)\|_{L^{\infty}}d\tau+\int_0^t\|\nabla
u^3(\tau)\|_{L^{\infty}}d\tau)\}
\end{split}
\end{equation}
whereas
\begin{equation*}
\begin{split}
\int_0^t\|\nabla u^h(\tau)\|_{L^{\infty}}d\tau \lesssim\|u^h\|_{L_t^1(\dot{B}_{p,1}^{1+\frac{3}{p}})}
\end{split}
\end{equation*}
and
\begin{equation*}
\begin{split}
\int_0^t\|\nabla u^3(\tau)\|_{L^{\infty}}d\tau\lesssim\|
u^3\|_{L_t^1(\dot{B}_{p,1}^{1+\frac{3}{p}})}^{1-\frac{1}{p}}\|
u^h\|_{L_t^1(\dot{B}_{p,1}^{1+\frac{3}{p}})}^{\frac{1}{p}},
\end{split}
\end{equation*}
using inequality (\ref{u^3 2}) with $m=r=\infty, q=\infty$.
Thus substituting the above estimates into the inequality
(\ref{transport mid eatimate}), we complete the proof of Proposition
\ref{transport equation eatimate}.
\end{proof}

Based on the above estimate and the estimate of the pressure
obtained by the same method as Proposition \ref{pressure}, we
complete the proof of Theorem \ref{main theorem} by the similar
arguments as the proof of Theorem \ref{main}, and omit the details.

\section*{Acknowledgements}
  This work is
  partially supported by  NSF of
China under Grants 11271322,  11331005 and 11271017, National Program for
Special Support of Top-Notch Young Professionals, Program for New Century
Excellent Talents in University NCET-11-0462, the Fundamental Research
Funds for the Central Universities (2012QNA3001).

\end{document}